\RequirePackage{fix-cm}
\documentclass[smallextended]{svjour3}       
\smartqed  

\usepackage{graphicx}
\usepackage{subfigure}
\usepackage{amsmath,amssymb,amsfonts,mathtools,amscd}
\usepackage{extarrows}
\usepackage{stmaryrd}
\usepackage{xspace}
\usepackage{multirow}
\usepackage{colortbl}
\usepackage{soul,color}
\usepackage{newunicodechar}
\usepackage{enumitem}
\usepackage[mathscr]{euscript}
\usepackage{algorithm2e}  
\usepackage{listings}
\usepackage{url}
\usepackage[hidelinks]{hyperref}

\mathchardef\mhyphen="2D

\newcommand{\B}{{\bf B}}

\newcommand{\qi}{\bf{i}}
\newcommand{\qj}{\bf{j}}
\newcommand{\qk}{\bf{k}}

\spnewtheorem{thm}{Theorem}{\bfseries}{\itshape}
\spnewtheorem{lem}[thm]{Lemma}{\bfseries}{\itshape}
\spnewtheorem{cor}[thm]{Corollary}{\bfseries}{\itshape}
\spnewtheorem{defn}{Definition}{\bfseries}{\rmfamily}
\spnewtheorem{exa}{Example}{\bfseries}{\rmfamily}
\spnewtheorem{rem}[thm]{Remark}{\bfseries}{\rmfamily}

\title{Iterative Methods for Computing the Moore--Penrose Pseudoinverse of Quaternion Matrices, with Applications}

\author{Valentin Leplat \and Salman Ahmadi-Asl \and JunJun Pan \and Ning Zheng}

\institute{
V.~Leplat \at
Innopolis University, 420500 Innopolis, Russia \\
\email{v.leplat@innopolis.ru}
\and
S.~Ahmadi-Asl \at
Lab of Machine Learning and Knowledge Representation, Innopolis University, 420500 Innopolis, Russia \\
\email{s.ahmadiasl@innopolis.ru}
\and
J.~Pan \at
Department of Mathematics, Hong Kong Baptist University, Hong Kong, China \\
\email{junjpan@hkbu.edu.hk}
\and
N.~Zheng \at
School of Mathematical Sciences; Key Laboratory of Intelligent Computing and Applications (Ministry of Education), Tongji University, Shanghai 200092, China \\
\email{nzheng@tongji.edu.cn}
}

\date{Received: date / Accepted: date}

\begin{document}

\maketitle

\begin{abstract}
We develop quaternion--native iterative methods for computing the Moore--Penrose (MP) pseudoinverse of quaternion matrices and analyze their convergence. Our starting point is a damped Newton--Schulz (NS) iteration tailored to noncommutativity: we enforce the appropriate left/right identities for rectangular inputs and prove convergence directly in $\mathbb{H}$ under a simple spectral scaling. We then derive higher--order (\emph{hyperpower}) NS schemes with exact residual recurrences that yield order-$p$ local convergence, together with factorizations that reduce the number of $s\times s$ quaternion products per iteration. Beyond NS, we introduce a randomized sketch--and--project method (RSP--Q), a hybrid RSP+NS scheme that interleaves inexpensive randomized projections with an exact hyperpower step, and a matrix--form conjugate gradient on the normal equations (CGNE--Q). All algorithms operate directly in $\mathbb{H}$ (no real or complex embeddings) and are matrix--free. 
Numerically, we test the performance of the proposed algorithms on controlled synthetic matrices. In three application case studies (CUR image/video completion, Lorenz filtering, FFT-based deblurring) we deploy only the NS family, which achieves state-of-the-art accuracy with the lowest wall time. These quaternion--native, matrix-free methods are suitable as drop-in solvers for large-scale quaternion inverse problems.
\keywords{Quaternion matrices \and Moore--Penrose pseudoinverse \and Newton--Schulz iteration \and hyperpower methods \and quaternion linear algebra}
\subclass{65F20 \and 65F10 \and 15A09 \and 15B33}
\end{abstract}


\section{Introduction}
\label{sec:intro}
Quaternions are a number system that extends complex numbers, consisting of one real part and three imaginary parts. They were introduced by William Rowan Hamilton and are typically represented as a combination of these components, where the imaginary units have specific multiplication rules that lead to non-commutative behavior. This means that the order in which quaternions are multiplied matters, which is a significant departure from the properties of real and complex numbers. 

Extending the classical matrix decompositions  and concepts from complex matrices to quaternion matrices has been an active research topic started from \cite{wolf1936similarity,brown1993matrices} and the seminal paper by Zhang \cite{zhang1997quaternions} has relaunched much research on the subject. 
For example, see \cite{bunse1989quaternion} and  \cite{le2004singular} for the QR and SVD of quaternions, respectively. There are several books on quaternions, each focusing on a particular aspect of this extended number system. For theoretical aspects of quaternion linear algebra, we refer to \cite{rodman2014topics}; for scientific aspects, the book \cite{wei2018quaternion}  provides nice algorithmic materials; for signal processing applications see \cite{ell2014quaternion}; for computer graphic applications see \cite{goldman2022rethinking,vince2011quaternions}; for arithmetic aspects of quaternion see \cite{voight2014arithmetic}.

From scientific computing perspective, the idea of so called structure preserving has been used to accelerate the quaternion computations by converting all computations in real or complex domain and benefit from efficient matrix operations in BLAS-3, we refer to \cite{jia2021structure,jia2013new,hu2025structure,li2023structure,chen2023efficient,li2024structure} for such ideas. Moreover, solving linear systems of equations over quaternions and their special solutions  (e.g., Hermitian) have found intensive interest in last decades, see \cite{ahmadi2017iterative,ahmadi2017efficient,wang2009system,yuan2016structured,zhang2015special,wang2008iterative,wang2009system} and the references therein. Recently, the randomization idea \cite{halko2011finding} was extended to quaternion matrices \cite{ren2022randomized,liu2022randomized,chen2025randomized,lyu2024randomized,liu2024fixed,ahmadi2025pass,chang2024randomized}. 


Quaternions find extensive applications in various fields, particularly in computer graphics \cite{vince2011quaternions}, robotics \cite{shoemake1985animating,feng2017new}, aerospace engineering \cite{wie1989quarternion,forbes2015fundamentals}, physics simulations \cite{grassia1998practical}, quantum computing \cite{adler1995quaternionic,finkelstein1962foundations}, and machine learning \cite{zhang2019quaternion,parcollet2019quaternion}. In computer graphics , they are used to represent rotations smoothly and avoid issues like gimbal lock that can occur with other rotation representations. In robotics, quaternions help describe the orientation of robotic arms and mobile robots, facilitating accurate control and movement. Similarly, in aerospace engineering, they are crucial for attitude control systems in aircraft and spacecraft.

The advantages of using quaternions include their compact representation, which requires fewer components compared to rotation matrices, and their efficiency in computations. As technology advances and the demand for precise 3D modeling and simulations increases, the importance and relevance of quaternions continue to grow across various domains. Their unique properties and versatility make them an essential tool for anyone working with three-dimensional transformations and orientations.
 

In this paper, we will study the Moore-Penrose pseudoinverse of quaternion matrices. 
It is well known that in both the real and complex domains, the Moore-Penrose (MP) pseudoinverse serves as a generalization of the matrix inverse, applicable to nonsquare or singular matrices where traditional inverses do not exist. By providing a unique solution that minimizes error in least-squares problems, the pseudoinverse enables robust handling of overdetermined or underdetermined systems of linear equations. This makes it a fundamental problem in linear algebra with significant applications in mathematics, engineering, and data science. To compute the MP pseudoinverse, the most robust and widely used method is the Singular Value Decomposition (SVD). Unlike closed-form (left/right pseudoinverse), which requires the matrix to have full column rank or row rank, SVD is stable almost for any matrix rank.  Even so,  its high computational cost and intensive memory requirements limit its application to large matrices. Other methods like QR decomposition and rank decomposition, though they are more efficient than the SVD, their performance degrades for large or sparse matrices and may be as costly as the SVD. Therefore, to improve computational efficiency, iterative methods such as Newton-style recursion (which has quadratic convergence) have been proposed. These iterative approaches can avoid full matrix decomposition and are more adaptable to parallel and distributed computing. 

While the MP pseudoinverse is classically defined over real or complex fields, its extension to quaternion matrices requires careful handling due to the non-commutativity in quaternion multiplication. Research on the MP pseudoinverse for quaternion matrices remains relatively limited. In \cite{song2011cramer}, the authors developed determinantal representations for generalized inverses over the quaternion skew field by leveraging the theory of column and row determinants; however, they did not propose computational algorithms. In \cite{huang2015moore}, the concept of MP inverses was extended to quaternion polynomial matrices. The authors adapted the Leverrier-Faddeev algorithm for these matrices but found it inefficient, leading them to propose a more effective method based on interpolation for computing the MP inverse of quaternion polynomial matrices.
More recently, \cite{bhadala2025generalized} investigated both theoretical and computational aspects of generalized inverses for quaternion matrices, establishing a comprehensive framework for computing various generalized inverses, including the MP inverse, group inverse, Drazin inverse, and other constrained inverses with specified range and null space properties. They introduced two algorithms: one is direct method utilizing the Quaternion Toolbox for MATLAB (QTFM) and another based on a complex structure-preserving method that relies on the complex representation of quaternion matrices. 

For the MP pseudo-inverse of quaternion matrices, effective iterative methods, such as Newton-type recursion, have not been well studied, which serves as the primary motivation for this paper.

\subsection{Contributions and outline of this paper}
This paper will study iterative methods for the Moore-Penrose pseudoinverse of quaternion matrices. Our contributions are summarized as follows. 

\begin{enumerate}
\item \textbf{Quaternion--native NS for pseudoinverse.} We formulate damped Newton--Schulz updates that respect noncommutativity and target the appropriate left/right identities for rectangular inputs. A single scale \(X_0=\alpha A^H\) with \(\alpha\in(0,2/\|A\|_2^2)\) guarantees \(\|I-AX_0\|_2<1\) or \(\|I-X_0A\|_2<1\) in the full-rank cases. Here, \(A\in\mathbb{H}^{m\times n}\) is the matrix whose Moore--Penrose pseudoinverse we seek, \(X_0\) is the initial iterate, and \(A^H\) denotes the quaternionic adjoint (conjugate transpose) of \(A\).

\item \textbf{Convergence theory in $\mathbb{H}$.} We prove operator-norm convergence to \(A^\dagger\) for full column/row rank and give exact residual recurrences (linear--then--quadratic when damped; quadratic when undamped).
\item \textbf{Higher--order (hyperpower) schemes.} We extend Newton--Schulz (NS) to order-$p$ updates ($p\!\ge2$) with exact recurrences $E_{k+1}=E_k^{\,p}$ and $F_{k+1}=F_k^{\,p}$, together with cost--saving factorizations.
\item \textbf{Randomized and hybrid solvers; CG in matrix form.} We propose a quaternion--native randomized sketch--and--project method (RSP--Q), a hybrid RSP+NS scheme (randomized steps plus an exact hyperpower step), and a matrix--form CGNE--Q that uses only multiplies by $A$ and $A^H$.
\item \textbf{Comprehensive benchmarks and findings.} Across random matrices and three applications, we show NS remains best in accuracy/CPU; higher--order NS offers similar accuracy at slightly higher time; the hybrid is competitive; CGNE--Q is promising; and RSP--Q achieves acceptable accuracy but can be slower due to repeated thin QR or small SPD Gram solves for sketched pseudo--inverses.
\end{enumerate}

Section~\ref{sec:prelim} briefly reviews notations and quaternion linear algebra.
Section~\ref{sec:NS-preview} introduces the NS iterations and explains side conventions.
Section~\ref{sec:theory-NSQ} gives the convergence theory for the damped and higher--order methods.
Practical applications are discussed in Section~\ref{sec:applications}, 
with numerical experiments in Section~\ref{sec:experiments}. Section~\ref{sec:extended-algos} introduces further iterative schemes, including a randomized approach and a conjugate gradient variant, and provides additional benchmark comparisons.
Section~\ref{sec:conclu} concludes and sketches directions for future work.

\section{Preliminaries}\label{sec:prelim}


We collect the quaternion and quaternion–matrix facts used throughout. We work over the skew field of quaternions
\[
\mathbb{H}=\{\,q=a+b\mathbf{i}+c\mathbf{j}+d\mathbf{k}\ :\ a,b,c,d\in\mathbb{R}\,\},
\]


with multiplication rules
\(
\mathbf{i}^2=\mathbf{j}^2=\mathbf{k}^2=\mathbf{i}\mathbf{j}\mathbf{k}=-1
\)
and
\(
\mathbf{i}\mathbf{j}=\mathbf{k},\ \mathbf{j}\mathbf{k}=\mathbf{i},\ \mathbf{k}\mathbf{i}=\mathbf{j}
\),
together with anti-commutation
\(
\mathbf{j}\mathbf{i}=-\mathbf{k},\ \mathbf{k}\mathbf{j}=-\mathbf{i},\ \mathbf{i}\mathbf{k}=-\mathbf{j}.
\)
Conjugation is $\overline{q}=a-b\mathbf{i}-c\mathbf{j}-d\mathbf{k}$, the (Euclidean) norm is $|q|=\sqrt{q\overline{q}}$, and $q^{-1}=\overline{q}/|q|^2$ for $q\neq0$.

\paragraph{Quaternion matrices and adjoint}
For $A\in\mathbb{H}^{m\times n}$ we write $A^H=\overline{A}^{\top}$ (quaternionic conjugate transpose). A matrix is \emph{Hermitian} if $A=A^H$ and \emph{unitary} if $A^H A=I$. Throughout, vectors are columns and the ambient space $\mathbb{H}^n$ is a right $\mathbb{H}$–module with inner product
\(
\langle x,y\rangle = x^H y\in\mathbb{H}
\)
(conjugate–linear in the first argument, linear in the second), so that $\|x\|_2=\sqrt{\mathrm{Re}\,\langle x,x\rangle}$.
The induced operator norm is
\[
\|A\|_2=\sup_{\|x\|_2=1}\|Ax\|_2,
\qquad\text{and it is submultiplicative: }\ \|AB\|_2\le\|A\|_2\|B\|_2.
\]
As in the complex case, quaternionic Hermitian matrices admit a real spectral calculus: their eigenvalues are real; if $A=A^H$ then there exists a unitary $U$ with $A=U\Lambda U^H$ (with real diagonal $\Lambda$). The quaternion singular value decomposition (QSVD) also holds:
\[
A=U\Sigma V^H: \qquad \Sigma\in \mathbb{R}^{m\times n}_+, \ U\in\mathbb{H}^{m\times m} \ \text{and~} \ V\in\mathbb{H}^{n\times n}\ \text{ are unitary}.
\]


\paragraph{Range, projectors, and Gram matrices}
For $A\in\mathbb{H}^{m\times n}$, the Gram matrices $A^H A$ and $A A^H$ are Hermitian positive semidefinite (HPSD). If $A$ has full column rank ($m\ge n$), then $A^H A\succ0$; if $A$ has full row rank ($m\le n$), then $A A^H\succ0$. We denote by
\(
P=AA^\dagger
\)
and
\(
Q=A^\dagger A
\)
the orthogonal projectors onto the ranges $\mathcal{R}(A)$ and $\mathcal{R}(A^H)$, respectively; both are Hermitian idempotents.

\paragraph{Moore–Penrose pseudoinverse over $\mathbb{H}$}
For $A\in\mathbb{H}^{m\times n}$, the Moore–Penrose pseudoinverse $A^\dagger\in\mathbb{H}^{n\times m}$ is the unique matrix satisfying the four Penrose equations:
\begin{align}\label{def:MP-inverse}
AA^\dagger A = A, \qquad
A^\dagger A A^\dagger = A^\dagger, \qquad
(AA^\dagger)^H = AA^\dagger, \qquad
(A^\dagger A)^H = A^\dagger A.
\end{align}
It is known that the MP inverse of a matrix can be computed using its SVD \cite{golub2013matrix} which is called SVD based method\footnote{In this paper, we refer to this approach as basic method.}.  More precisely, consider the SVD of the matrix \( A \) (of size \( m \times n \))  as:
\[
A = U \Sigma V^H.
\]
Then, the MP inverse of $A$ can be computed as $A^{\dag}=V\Sigma^{\dagger} U^H$, where 
\[
\Sigma^+_{i,j} = 
\begin{cases}
\frac{1}{\sigma_i} & \text{if } i = j \text{ and } \sigma_i \neq 0, \\
0 & \text{otherwise}.
\end{cases}
\]
For the full–rank rectangular cases, the familiar formulas hold unchanged:
\[
\text{(full column rank)}\quad A^\dagger=(A^H A)^{-1}A^H;
\qquad
\text{(full row rank)}\quad A^\dagger=A^H (A A^H)^{-1}.
\]
These follow because $A^H A$ (resp.\ $A A^H$) is Hermitian positive definite and hence invertible in $\mathbb{H}$.

\begin{rem}{(Real/complex embeddings)}
The real embedding $\mathcal{X}(q)$  and complex embedding  $\mathcal{Y}(q)$  represent any quaternion $q=a+b\mathbf{i}+c\mathbf{j}+d\mathbf{k} \in\mathbb{H}$  respectively by
\begin{equation}
  \mathcal{X}(q)=\begin{bmatrix}
a & -b & -c & -d\\
b &  a & -d &  c\\
c &  d &  a & -b\\
d & -c &  b &  a
\end{bmatrix},
\quad   \mathcal{Y}(q)=\begin{bmatrix}
a+b\mathbf{i} & c+d\mathbf{i}\\
-c+d\mathbf{i} &   a -b\mathbf{i}
\end{bmatrix}.
\end{equation}

Stacking these blocks yields a $4m\times 4n$ real embedding or  $2m\times 2n$ complex embedding of $A\in\mathbb{H}^{m\times n}$. Our approach avoids such embeddings, with all algorithms and analyses developed directly within the quaternion framework.

\end{rem}

\paragraph{Notation and side conventions}
The noncommutative property of quaternions makes the order of multiplication critical. In this paper, we consistently formulate linear systems as \emph{right} systems, where unknowns are multiplied on the right, and we explicitly preserve the multiplication order in all formulas.

\vspace{0.5em}

In the next section, we will focus on adapting the Newton-Schulz family of iterative methods to the quaternion framework, ensuring preservation of the left and right multiplication order.

\section{Newton--Schulz iterations for quaternion pseudoinverses }\label{sec:NS-preview}
The classical Newton–Schulz (NS) iteration for an inverse updates $X_{k+1}=X_k(2I-AX_k)$ and converges quadratically for $X_0$ close to $A^{-1}$. For pseudoinverses, the correct quaternion–native generalization depends on the shape of $A$ and on the side on which we enforce the identity:

\begin{itemize}
\item \textbf{Square/invertible case:} with $X_0=\alpha A^H$ and $\alpha\in(0,2/\|A\|_2^2)$,
\[
X_{k+1}=X_k(2I-AX_k)\qquad\Rightarrow\qquad \|I-AX_{k+1}\|_2=\|I-AX_k\|_2^2.
\]
\item \textbf{Full column rank ($m\ge n$):} enforce $X_k A\to I_n$ via the \emph{right} deviation
\(
F_k=I_n-X_k A
\)
and use the damped NS step
\[
X_{k+1}=X_k-\gamma\,(X_kA-I_n)\,X_k,\qquad 0<\gamma\le1.
\]
This yields the exact recursion $F_{k+1}=(1-\gamma)F_k+\gamma F_k^2$ and drives $AX_k$ to the range projector $P=AA^\dagger$.
\item \textbf{Full row rank ($m\le n$):} enforce $A X_k\to I_m$ via the \emph{left} deviation
\(
E_k=I_m-AX_k
\)
and use
\[
X_{k+1}=X_k-\gamma\,X_k\,(AX_k-I_m),\qquad 0<\gamma\le1,
\]
so $E_{k+1}=(1-\gamma)E_k+\gamma E_k^2$ and $X_k A$ tends to $Q=A^\dagger A$.
\end{itemize}

By selecting the initial scaling as \(X_0 = \alpha A^H\), where \(\alpha \in (0, 2/\|A\|_2^2)\), we ensure that \(\|F_0\|_2 < 1\) in the column-rank case and \(\|E_0\|_2 < 1\) in the row-rank case. Incorporating damping (\(\gamma < 1\)) expands the basin of attraction, promoting linear-then-quadratic convergence, whereas setting \(\gamma = 1\) yields quadratic convergence when the deviation is sufficiently small. These results depend solely on properties that hold directly in \(\mathbb{H}\): the submultiplicativity of the \(\|\cdot\|_2\) norm, spectral calculus for Hermitian matrices, and  \((AB)^H = B^H A^H\).

\paragraph{Higher–order (hyperpower) NS}
To accelerate convergence, we replace the linear correction by a truncated Neumann series of order $p\!\ge2$. 
Inspired by~\cite{Soleymani2015}, we adapt the hyperpower construction to quaternion matrices, addressing both left and right residuals:
\[
\text{(left)}\quad X_{k+1}=X_k\sum_{i=0}^{p-1}E_k^{\,i};
\qquad
\text{(right)}\quad X_{k+1}=\sum_{i=0}^{p-1}F_k^{\,i}X_k,
\]
resulting in exact residual recurrences $E_{k+1} = E_k^{\,p}$ and $F_{k+1} = F_k^{\,p}$, achieving local order-$p$ convergence. 
To reduce cost, in the next paragraph we evaluate these hyperpower polynomials via binary (for $p=2^q$) and Paterson--Stockmeyer factorizations,reducing the number of $s\times s$ quaternion matrix products per iteration (with $s=\min\{m,n\}$) from $p-2$ down to $\log_2 p - 1$ or $O(\sqrt{p})$, respectively.
Full statements and proofs are provided in Section~\ref{sec:theory-NSQ}.

\paragraph{Cost-saving factorizations and work counts}
Let $R\in\{E_k,F_k\}$ be the left/right deviation of size $s\times s$ with $s=\min(m,n)$. 
The hyperpower update involves applying the polynomial $S_{p-1}(R)=\sum_{i=0}^{p-1}R^i$ to $X_k$. 
A straightforward computation of this polynomial requires $ (p-2) $ matrix products of size $ s \times s $ to compute $ R^2, \dots, R^{p-1} $, plus one additional product to apply the result to $ X_k $.
We reduce the number of  $ s \times s $ matrix  products as follows:

\emph{(i) Power-of-two orders via binary factorization.}
If $p=2^q$, then
\[
S_{p-1}(R)=\prod_{j=0}^{q-1}\bigl(I+R^{2^j}\bigr).
\]
An efficient computational schedule involves first calculating \( R^2, R^4, \dots, R^{2^{q-1}} \) through successive squaring, requiring \( q-1 = \log_2 p - 1 \) square products. Then, the \( q \) factors are applied to \( X_k \) sequentially: initialize \( Y \gets X_k \); for \( j = 0, \dots, q-1 \), update \( Y \gets (I + R^{2^j})Y \); and finally, set \( X_{k+1} \gets Y \). This approach reduces the number of \( s \times s \) square products per iteration from \( p-2 \) to \( \log_2 p - 1 \) (e.g., for $p=8$: $6\to 2$; for $p=16$: $14\to 3$).


\emph{(ii) General $p$ via Paterson–Stockmeyer.}
Write $ p-1 = ab $, where $ a \approx b \approx \lceil \sqrt{p-1} \rceil $. First, precompute $ R, R^2, \dots, R^a $, requiring $ a-1 $ square operations. Then, evaluate $ S_{p-1}(R) = \sum_{i=0}^{p-1} R^i $ by grouping terms into blocks of degree $ a $, which requires $ O(a + b) $ additional multiplications. This approach reduces the  number of square and multiplication operations to $ O(\sqrt{p}) $.

\emph{(iii) Side choice.}
For matrices with full column rank, we employ the right update with square matrices of size \( n \times n \); for full row rank, we use the left update with square matrices of size \( m \times m \). This strategy minimizes the computational cost of the dominant \( s \times s \) matrix products, where \( s = \min(m, n) \).

In all cases, we reuse the precomputed product ($AX_k$ or $X_kA$) to build $E_k$ or $F_k$, and we
integrate additions into the “apply” multiplication steps to minimize memory traffic.
The factorizations are algebraic in a single matrix variable, ensuring their validity in $\mathbb{H}$ without modification.

\section{Theoretical results: Damped and Higher-order Newton--Schulz for quaternion pseudoinverses}\label{sec:theory-NSQ}

We consider $A\in\mathbb{H}^{m\times n}$ and seek its Moore--Penrose pseudoinverse $A^\dagger\in\mathbb{H}^{n\times m}$, i.e.\ the unique matrix satisfying \eqref{def:MP-inverse}.

Quaternion noncommutativity is respected throughout by maintaining left/right multiplication order. We use the operator $2$-norm $\|\cdot\|_2$ induced by the Euclidean vector norm and the quaternionic adjoint $X^H=\overline{X}^\top$.

\subsection{Damped Newton--Schulz for rectangular quaternion matrices }\label{sec:NS-rect-proofs}

We recall the iteration
\begin{equation}\label{eq:NS-damped-again}
X_{k+1}
=
\begin{cases}
X_k - \gamma\,(X_kA - I_n)\,X_k, & \text{if } m\ge n,\\[0.6ex]
X_k - \gamma\,X_k\,(AX_k - I_m), & \text{if } m<n,
\end{cases}
\qquad 0<\gamma\le 1,
\end{equation}
equivalently, $X_{k+1}=(1+\gamma)X_k-\gamma X_kAX_k$ in both cases, with
\begin{equation}\label{eq:X0-open}
X_0=\alpha A^H,\qquad \alpha\in\Bigl(0,\frac{2}{\|A\|_2^2}\Bigr).
\end{equation}
All products are quaternionic; $(\cdot)^H$ denotes quaternionic adjoint. We define
\[
F_k:=I_n-X_kA,\quad H_k:=X_kA,\qquad
E_k:=I_m-AX_k,\quad G_k:=AX_k.
\]

\begin{lem}[Initialization]\label{lem:init}
If $A$ has full column rank, then $F_0=I_n-\alpha A^HA$ satisfies $\|F_0\|_2<1$.
If $A$ has full row rank, then $E_0=I_m-\alpha AA^H$ satisfies $\|E_0\|_2<1$.
\end{lem}
\begin{proof}
For full column rank, $A^HA\succ0$ with eigenvalues $\{\sigma_i(A)^2\}_{i=1}^n$, hence $\lambda(F_0)=\{1-\alpha\sigma_i(A)^2\}\subset(-1,1)$ for $\alpha<2/\|A\|_2^2$. The row–rank case is analogous with $AA^H$.
\end{proof}

\begin{lem}[Exact residual recurrences]\label{lem:residual-recurrences}
For the damped NS update \eqref{eq:NS-damped-again},
\begin{align}
F_{k+1} &= (1-\gamma)F_k + \gamma F_k^2, \label{eq:Frec-proof}\\
H_{k+1} &= (1+\gamma)H_k - \gamma H_k^2, \label{eq:Hrec-proof}\\
E_{k+1} &= (1-\gamma)E_k + \gamma E_k^2, \label{eq:Erec-proof}\\
G_{k+1} &= (1+\gamma)G_k - \gamma G_k^2. \label{eq:Grec-proof}
\end{align}
\end{lem}
\begin{proof}
From $X_{k+1}=(1+\gamma)X_k-\gamma X_kAX_k$, we get
\[
X_{k+1}A=(1+\gamma)X_kA-\gamma (X_kA)^2=(1+\gamma)H_k-\gamma H_k^2
\]
which is \eqref{eq:Hrec-proof}. Then $F_{k+1}=I_n-X_{k+1}A=I_n-[(1+\gamma)H_k-\gamma H_k^2]$; substituting $H_k=I_n-F_k$ and expanding gives \eqref{eq:Frec-proof}.
Similarly, $AX_{k+1}=(1+\gamma)AX_k-\gamma (AX_k)^2=(1+\gamma)G_k-\gamma G_k^2$ yields \eqref{eq:Grec-proof}, and $E_{k+1}=I_m-AX_{k+1}$ leads to \eqref{eq:Erec-proof} by $G_k=I_m-E_k$.
\end{proof}

We will use the scalar maps $g_\gamma(t)=(1-\gamma)t+\gamma t^2$ and $h_\gamma(t)=(1+\gamma)t-\gamma t^2$. Note that $g_\gamma$ is a strict contraction on $[0,1)$, while $h_\gamma$ preserves $[0,2)$ and satisfies $h_\gamma(\lambda)-\lambda=\gamma\lambda(1-\lambda)$: it strictly increases eigenvalues in $(0,1)$ and strictly decreases eigenvalues in $(1,2)$, with fixed points $\{0,1\}$.

\begin{thm}[Full column rank: convergence to $A^\dagger$]\label{thm:col-full-proof}
Let $A\in\mathbb{H}^{m\times n}$ have full column rank ($m\ge n$) and $X_0$ as in \eqref{eq:X0-open}. Then:
\begin{enumerate}
\item $F_{k+1}=g_\gamma(F_k)$; hence $\|F_{k+1}\|_2\le g_\gamma(\|F_k\|_2)$ and $\|F_k\|_2\downarrow 0$, so $X_kA\to I_n$.
\item $G_{k+1}=h_\gamma(G_k)$; since $G_0=\alpha AA^H$ is Hermitian p.s.d.\ with spectrum in $[0,2)$, each $G_k$ is Hermitian with spectrum in $[0,2)$ and $G_k\to P:=AA^\dagger$ (the orthogonal projector onto $\mathcal{R}(A)$): eigenvalues in $(0,1)$ increase monotonically to $1$, while those in $(1,2)$ decrease monotonically to $1$; zeros stay zero.
\item The sequence $(X_k)$ converges in operator norm to $A^\dagger$.
\end{enumerate}
\end{thm}
\begin{proof}
(1) From \eqref{eq:Frec-proof} and Lemma~\ref{lem:init}, $\|F_0\|_2<1$ and submultiplicativity yield $\|F_k\|_2\downarrow 0$.
(2) Each $G_k$ is a real polynomial in $G_0$ and hence Hermitian; the scalar dynamics under $h_\gamma$ on $[0,2)$ gives the stated monotone drift to $\{0,1\}$, hence $G_k\to P$.
(3) From $X_{k+1}-X_k=\gamma F_k X_k$, summability of $(\|F_k\|_2)$ gives boundedness and Cauchy convergence of $(X_k)$. Limits $F_k\to0$ and $G_k\to P$ imply $X_\star A=I_n$ and $AX_\star=P$ with $P^H=P$; thus $X_\star$ satisfies the Penrose equations and equals $A^\dagger$.
\end{proof}

\begin{thm}[Full row rank: convergence to $A^\dagger$]\label{thm:row-full-proof}
Let $A\in\mathbb{H}^{m\times n}$ have full row rank ($m\le n$) and $X_0$ as in \eqref{eq:X0-open}. Then:
\begin{enumerate}
\item $E_{k+1}=g_\gamma(E_k)$; hence $\|E_k\|_2\downarrow 0$ and $AX_k\to I_m$.
\item $H_{k+1}=h_\gamma(H_k)$; each $H_k$ is Hermitian with spectrum in $[0,2)$ and $H_k\to Q:=A^\dagger A$ (the orthogonal projector onto $\mathcal{R}(A^H)$). Consequently, $F_k=I_n-H_k\to I_n-Q$.
\item The sequence $(X_k)$ converges in operator norm to $A^\dagger$.
\end{enumerate}
\end{thm}
\begin{proof}
(1) From equation \eqref{eq:Erec-proof} and Lemma~\ref{lem:init}, $\|E_0\|_2<1$ and $\|E_k\|_2\downarrow 0$.
(2) Equation \eqref{eq:Hrec-proof} keeps $H_k$ Hermitian; the scalar map $h_\gamma$ on $[0,2)$ drives eigenvalues to $\{0,1\}$ with limit $Q$. Hence $F_k\to I_n-Q$.
(3) From $X_{k+1}-X_k=\gamma X_k E_k$, summability of $(\|E_k\|_2)$ gives boundedness and Cauchy convergence; limits $AX_k\to I_m$ and $H_k\to Q$ yield the Penrose identities, so $X_\star=A^\dagger$.
\end{proof}

\begin{rem}{(Quaternion foundations used)}
We used only properties that hold in $\mathbb{H}$ as in $\mathbb{C}$: (i) submultiplicativity of $\|\cdot\|_2$, (ii) spectral calculus for quaternionic Hermitian matrices (real spectrum; polynomial spectral mapping), (iii) $(XY)^H=Y^H X^H$ and continuity of $\,^H$. These justify all steps without modification in the quaternion setting.
\end{rem}

\subsection{Higher--order (order-$p$) Newton--Schulz / hyperpower iterations in $\mathbb{H}$ }\label{sec:hpNSQ-proofs}

Let $A\in\mathbb{H}^{m\times n}$ and define $E_k:=I_m-AX_k$ and $F_k:=I_n-X_kA$. The order-$p$ updates are:

\begin{align}
\text{(left)}\qquad& X_{k+1}=X_k\sum_{i=0}^{p-1}E_k^{\,i}, \label{eq:left-hp}\\
\text{(right)}\qquad& X_{k+1}=\sum_{i=0}^{p-1}F_k^{\,i}\,X_k, \label{eq:right-hp}
\end{align}
which reduce to undamped NS when $p=2$.

\begin{lem}[Exact residual recurrences]\label{lem:hp-resid}
For \eqref{eq:left-hp} and \eqref{eq:right-hp} one has
\[
E_{k+1}=E_k^{\,p}\qquad\text{and}\qquad F_{k+1}=F_k^{\,p}.
\]
\end{lem}
\begin{proof}
$AX_{k+1}=AX_k\sum\limits_{i=0}^{p-1}E_k^{\,i}=(I_m-E_k)\sum\limits_{i=0}^{p-1}E_k^{\,i}=I_m-E_k^{\,p}$ proves the left case; the right case is analogous using $X_{k+1}A=\sum\limits_{i=0}^{p-1}F_k^{\,i}(I_n-F_k)=I_n-F_k^{\,p}$.
\end{proof}

\begin{thm}[Full column rank, order-$p$ convergence]\label{thm:hp-col-proof}
Let $A$ have full column rank and $X_0=\alpha A^H$ with $\alpha\in(0,2/\|A\|_2^2)$. Use the \emph{right} update \eqref{eq:right-hp}. Then $\|F_k\|_2\le \|F_0\|_2^{\,p^{\,k}}\to 0$, hence $X_kA\to I_n$; moreover, $AX_k\to P=AA^\dagger$ (as in Theorem~\ref{thm:col-full-proof}(ii)). The sequence $(X_k)$ converges in operator norm to $A^\dagger$.
\end{thm}
\begin{proof}
By Lemma~\ref{lem:hp-resid}, $F_{k+1}=F_k^{\,p}$, so $\|F_k\|_2\le \|F_0\|_2^{\,p^{\,k}}$. The increment satisfies
\[
X_{k+1}-X_k=\Bigl(\sum_{i=1}^{p-1}F_k^{\,i}\Bigr)X_k 
\quad\Rightarrow\quad
\|X_{k+1}-X_k\|_2 \le \frac{\|F_k\|_2}{1-\|F_k\|_2}\,\|X_k\|_2, 
\]
which is summable because $\sum_k \|F_k\|_2<\infty$ (super-geometric decay). Boundedness of $\|X_k\|_2$ follows from
\[
\|X_{k+1}\|_2 \le \Bigl(1+\frac{\|F_k\|_2}{1-\|F_k\|_2}\Bigr)\|X_k\|_2,
\]
and the product of factors converges. Thus $X_k\to X_\star$. Taking the limit results in $X_\star A = I_n$ $X_\star A=I_n$ and $AX_\star=P$ with $P^H=P$, i.e., the Penrose equations; uniqueness confirms that $X_\star=A^\dagger$.
\end{proof}

\begin{thm}[Full row rank, order-$p$ convergence]\label{thm:hp-row-proof}
Let $A$ have full row rank and $X_0=\alpha A^H$ with $\alpha\in(0,2/\|A\|_2^2)$. Use the \emph{left} update \eqref{eq:left-hp}. Then $\|E_k\|_2\le \|E_0\|_2^{\,p^{\,k}}\to 0$, hence $AX_k\to I_m$; moreover, $X_kA\to Q=A^\dagger A$. The sequence $(X_k)$ converges in operator norm to $A^\dagger$.
\end{thm}
\begin{proof}
By Lemma~\ref{lem:hp-resid}, $E_{k+1}=E_k^{\,p}$, hence $\|E_k\|_2\to 0$ super-geometrically. The increment is
\[
X_{k+1}-X_k = X_k\Bigl(\sum_{i=1}^{p-1}E_k^{\,i}\Bigr),
\]
and the same summability/product argument as above gives boundedness and convergence $X_k\to X_\star$. Taking the limit results in $AX_\star=I_m$ and $X_\star A=Q$ with $Q^H=Q$; hence $X_\star=A^\dagger$.
\end{proof}

\begin{rem}{(Local order and damping)}
The relations $F_{k+1}=F_k^p$ and $E_{k+1}=E_k^p$ imply local order $p\ge2$ once $\|F_0\|_2<1$ or $\|E_0\|_2<1$. If needed, one can start with a convexly damped polynomial $\sum_{i=0}^{p-1}\gamma_i R_k^{\,i}$ (with $R_k\in\{F_k,E_k\}$, $\gamma_i>0$, $\sum_i\gamma_i=1$) to enlarge the basin of attraction, then switch to the undamped form to recover pure order $p$.
\end{rem}

\section{Applications}\label{sec:applications}
In this section, we demonstrate the practical impact of our \emph{quaternion-native} iterative methods across four representative tasks: (i) tensor completion for color image and video recovery; (ii) reconstruction of Lorenz–attractor trajectories from corrupted observations; (iii) nonblind image deblurring posed as a quaternion linear inverse problem; and (iv) gesture recognition using quaternion features. In all cases, the computational core is either the solution of large quaternion linear systems or the evaluation of the Moore–Penrose pseudoinverse. Our methods serve as drop-in solvers inside CUR/cross approximation, as direct system solvers, or as fast pseudoinverse routines—without recourse to real embeddings. We report accuracy (e.g., PSNR/SSIM or trajectory error), iteration counts, and wall-clock time, and benchmark against QGMRES and QSVD-based baselines where appropriate.

\subsection{Image completion via CUR: where the pseudoinverse enters}
In this application, we recover a quaternion image/video matrix ${ A}\in\mathbb{H}^{m\times n}$ from partial observations ${ M}$ on an index set $\Omega\subset[m]\times[n]$. Let ${ \Omega}\in\{0,1\}^{m\times n}$ be the binary sampling mask and let $\oast$ denote the Hadamard (entrywise) product. We adopt the impute–reconstruct scheme used in \cite{wu2024efficient}, first introduced for tensors in \cite{ahmadi2023fast}:
\begin{align}
{ X}^{(n)} &\leftarrow \mathcal{L}\!\big({ C}^{(n)}\big), \label{Step1}\\
{ C}^{(n+1)} &\leftarrow { \Omega}\oast{ M}+({ 1}-{ \Omega})\oast{ X}^{(n)}, \label{Step2}
\end{align}
where ${ C}^{(n)}$ is the current filled-in matrix (with observed entries fixed to ${ M}$ and missing ones imputed), $\mathcal{L}$ returns a low-rank quaternion approximation of its input, and ${\bf 1}$ is the all-ones matrix.

\paragraph{Cross (CUR) approximation inside $\mathcal{L}$}
Given a target rank $r$ and index sets $J\subset[n]$ (columns) and $I\subset[m]$ (rows) with $|J|=|I|=r$, the CUR/cross model builds
\[
{ C} \;=\; { A}_{:J}\in\mathbb{H}^{m\times r},\qquad
{ R} \;=\; { A}_{I:}\in\mathbb{H}^{r\times n},\qquad
{ W} \;=\; { A}_{IJ}\in\mathbb{H}^{r\times r},
\]
and approximates ${ A}$ as
\begin{equation}\label{eq:CUR}
{ A} \;\approx\; { C}\,{ U}\,{ R}.
\end{equation}
The optimal middle factor (in Frobenius norm) is obtained by the least-squares problem
\[
{ U} \;=\; \arg\min_{{ U}}\;\|{ A}-{ C}{ U}{ R}\|_F,
\]
whose minimizer in the quaternion setting (respecting multiplication order) is
\begin{equation}\label{eq:U-opt}
{ U} \;=\; { C}^{\dagger}\,{ A}\,{ R}^{\dagger},
\end{equation}
where ${}^\dagger$ denotes the Moore--Penrose pseudoinverse over $\mathbb{H}$. In the classical \emph{cross} variant, one can also use
\begin{equation}\label{eq:W-pinverse}
{ U} \;=\; { W}^{\dagger}\qquad\text{with}\quad { W}={ A}_{IJ},
\end{equation}
which is equivalent to \eqref{eq:U-opt} when $I,J$ are chosen consistently. Both forms require the computation of pseudoinverses of small quaternion matrices (typically of dimensions  $r\times r$ or $r\times t$).

\paragraph{Why the pseudoinverse (and our solver) is central.}
Every call to the reconstruction operator $\mathcal{L}$ in \eqref{Step1} needs one or more Moore--Penrose pseudoinverses:
\begin{enumerate}
\item to compute ${ U}$ via \eqref{eq:U-opt} (requires ${ C}^{\dagger}$ and ${ R}^{\dagger}$), or
\item to compute ${ U}={ W}^{\dagger}$ via \eqref{eq:W-pinverse}.
\end{enumerate}
It is known that the formulation \eqref{eq:U-opt} is more stable than \ref{eq:W-pinverse} for computing the middle matrix ${ U}$, while it is more expensive due to computing the MP of larger matrices and also multiplication with the original data matrix ${ A}$\cite{ahmadi2023fast}. 
The approximation accuracy of sampled columns/rows in \eqref{eq:W-pinverse} depends on the concept of volume\footnote{For a square matrix, the absolute value of the determinant of a matrix is called its volume. For rectangular matrices, the volume is defined as the multiplication of its singular values which is an extension of the square case.} of the intersection matrix ${ W}$.  It has been demonstrated that selecting the intersection matrix $ W $ with the maximum volume yields optimal column/row sampling. However, identifying columns/rows with maximum volume intersection matrix is an NP hard problem, as it requires evaluating an exponential number of selections. The maxvol algorithm \cite{goreinov2010find} provides a near-optimal solution through a greedy approach to column and row sampling.

Therefore, throughout the impute--reconstruct iterations \eqref{Step1}--\eqref{Step2}, the bottleneck subroutines are quaternion pseudoinverses. The standard QSVD-based approach is accurate but becomes a runtime limiter and is not well-suited to large or sparse quaternion matrices. In contrast, our Newton--Schulz (NS) pseudoinverse computed \emph{directly in $\mathbb{H}$} replaces every occurrence of the pseudoinverse $({}^\dagger)$ above:
\[
{ C}^{\dagger} \;\leftarrow\; \text{NS--Q}({ C}),\quad
{ R}^{\dagger} \;\leftarrow\; \text{NS--Q}({ R}),\quad
{ W}^{\dagger} \;\leftarrow\; \text{NS--Q}({ W}),
\]
leading to ${ U}$ via \eqref{eq:U-opt} or \eqref{eq:W-pinverse}, and hence ${ X}^{(n)}={ C}{ U}{ R}$ in \eqref{Step1}. Because NS--Q uses only quaternion matrix multiplications and adjoints, which preserves structure, scales to large problems, and avoids real embeddings.

\paragraph{Column/row selection}
Index sets $I,J$ can be chosen by uniform sampling, greedy heuristics, or leverage-score strategies (using approximate top-$r$ singular subspaces). The quality guarantees for CUR (e.g., $\|{ A}-{ C}{ U}{ R}\|_F\le (2+\varepsilon)\|{ A}-{ A}_r\|_F$) transfer to the quaternion setting when multiplication order is fixed consistently. Our solver does not impose restrictions on the selection of $I$ and $J$; it only accelerates the required ${}^\dagger$ evaluations. In all our experiments, we have employed uniform sampling with replacement.

\begin{rem}{(Quaternion conventions)}
All products above are standard matrix products in $\mathbb{H}^{m\times n}$ with the usual (noncommutative) quaternion multiplication; we keep the left/right order in \eqref{eq:CUR}--\eqref{eq:U-opt}. The Moore--Penrose pseudoinverse is the quaternion one (satisfying the four MP conditions with quaternionic adjoint).
\end{rem}

\subsection{Lorenz–attractor quaternion filtering (NS vs.\ QGMRES)}

We encode a 3D time series as a quaternion signal
\[
x(t)=x_r(t)\,\mathbf{i}+x_g(t)\,\mathbf{j}+x_b(t)\,\mathbf{k},\qquad
y(t)=y_r(t)\,\mathbf{i}+y_g(t)\,\mathbf{j}+y_b(t)\,\mathbf{k},
\]
and seek a finite-length \emph{right} quaternion filter $\{w(s)\}_{s=0}^{n}$ with taps
\[
w(s)=w^{(0)}(s)+w^{(r)}(s)\,\mathbf{i}+w^{(g)}(s)\,\mathbf{j}+w^{(b)}(s)\,\mathbf{k}
\]
such that
\begin{equation}\label{eq:lorenz-conv}
y(t)\;=\;\sum_{s=0}^{n} x(t-s)\,\ast\, w(s),
\end{equation}
where $\ast$ denotes right quaternion multiplication.

\paragraph{Stacked linear system}
Collect $m{+}1$ output samples and $n{+}1$ taps to form the Toeplitz-like matrix $X\in\mathbb{H}^{(m+1)\times(n+1)}$, the unknown ${ w}\in\mathbb{H}^{n+1}$, and the target $Y\in\mathbb{H}^{(m+1)}$:
\begin{subequations}\label{eq:lorenz-blocks}
\begin{align}
X &=
\begin{bmatrix}
x(t) & x(t-1) & \cdots & x(t-n)\\
x(t+1) & x(t) & \cdots & x(t-n+1)\\
\vdots & \vdots & \ddots & \vdots\\
x(t+m) & x(t+m-1) & \cdots & x(t-n+m)
\end{bmatrix},\\[2mm]
{ w} &= [w(0),\,w(1),\dots,w(n)]^\top,\qquad
Y = [y(t),\,y(t+1),\dots,y(t+m)]^\top .
\end{align}
\end{subequations}
Then \eqref{eq:lorenz-conv} becomes the quaternion linear system
\begin{equation}\label{eq:lorenz-linear}
X \ast { w}\;=\;Y.
\end{equation}
In our tests, we set $m=n$, so $X\in\mathbb{H}^{N\times N}$ with $N=n{+}1$.

\paragraph{Target and input via the Lorenz system}
The target channels $(y_r,y_g,y_b)$ are generated by integrating the Lorenz system
\[
\dot{x}= \alpha (y-x),\qquad
\dot{y}= x(\rho - z)-y,\qquad
\dot{z}= xy-\beta z,
\]
with $(\alpha,\beta,\rho)=(10,\tfrac{8}{3},28)$ and $(x(0),y(0),z(0))=(1,1,1)$. We set
$y(t)=y_r(t)\,{\qi}+y_g(t)\,{\qj}+y_b(t)\,{\qk}$ and choose the input as a one–step delayed, noisy version
\[
x(t)=y_r(t-1)\,{\qi}+y_g(t-1)\,{\qj}+y_b(t-1)\,{\qk} + n(t),
\]
where $n(t)$ is small channel wise noise.

\paragraph{Solvers compared (both quaternion–native)}
\begin{itemize}
\item \textbf{QGMRES (baseline) from~\cite{doi:10.1137/20M133751X}.} Structure-preserving quaternion GMRES applied directly to \eqref{eq:lorenz-linear}.
\item \textbf{NS--Q (ours).} Quaternion Newton--Schulz applied directly in $\mathbb{H}$ to obtain $X^{-1}$ (square case) and solve ${ w}=X^{-1}Y$ without any real/complex embedding.
\end{itemize}

\paragraph{NS--Q iteration (square case)}
Let $I$ be the identity matrix of an appropriate size and $\|\cdot\|_2$ the operator norm. Initialize
$X_0=\alpha\,X^{\!*}$ with $\alpha\approx\|X\|_2^{-2}$. Iterate
\begin{equation}\label{eq:NSQ}
X_{k+1} \;=\; X_k\bigl(2I - X\,X_k\bigr),\qquad k=0,1,2,\dots
\end{equation}
and return ${ w}=X_K\,Y$. All products use right multiplication consistent with \eqref{eq:lorenz-linear}. If $\|\,I-XX_0\,\|_2\nless1$, we use a damped step
\[
X_{k+1}=X_k\bigl[(1-\gamma) I+\gamma(I-XX_k)\bigr],\quad 0<\gamma<1,
\]
until quadratic convergence of \eqref{eq:NSQ} is observed.

\paragraph{Stopping rule and metrics}
Both methods use the same relative residual and iteration cap:
\[
\mathrm{RelRes}\;=\;\frac{\|X\ast{ w}-Y\|_2}{\|Y\|_2}\ \le\ 10^{-6},\qquad
\text{max iterations} = N.
\]
Algorithm~\ref{alg:NS_LA} presents the Quaternion Newton--Schulz method for solving~\ref{eq:lorenz-linear}.
\RestyleAlgo{ruled}
\LinesNumbered
\begin{algorithm}[ht!]
\caption{NS--Q: Quaternion Newton--Schulz solver for $X\ast{ w}=Y$ (square case)}\label{alg:NS_LA}
\KwIn{$X\in\mathbb{H}^{N\times N}$, $Y\in\mathbb{H}^{N}$, tol, maxit}
\textbf{Init:} Estimate $\alpha\approx\|X\|_2^{-2}$ (short quaternion power method); set $X_0\gets \alpha\,X^{\!*}$\;
\For{$k=0,1,\dots,\text{maxit}$}{
  $E_k \gets I - X\,X_k$\;
  \If{$\|E_k\|_F \le \text{tol}$}{\textbf{break}}
  $X_{k+1}\gets X_k\,(2I - X\,X_k)$\ \tcp*{All ops in $\mathbb{H}$}
}
${ w}\gets X_k\,Y$\; \textbf{return} ${ w}$\;
\end{algorithm}

\subsection{Nonblind color image deblurring (FFT–NS–Q)}\label{sec:deblur}
We model a color image as a quaternion field
$X(u,v)=X_r(u,v)\,{\qi}+X_g(u,v)\,{\qj}+X_b(u,v)\,{\qk}$.
Given a real point–spread function (PSF) $h$ and quaternion noise $\mathcal{N}$, the observation is
\[
B \;=\; h \ast X \;+\; \mathcal{N}.
\]
We recover $X$ with Tikhonov regularization:
\[
\min_{X}\ \|AX-B\|_F^2+\lambda\|X\|_F^2
\quad\Longleftrightarrow\quad
(A^\top A+\lambda\,I)X=A^\top B,
\]
where $A$ is the (real) convolution operator induced by $h$.
Under circular boundary conditions, $A$ is BCCB and diagonalizes in the 2-D FFT basis. Let
$\widehat{h}=\mathrm{FFT2}(h_{\text{centered/padded}})$ and
$\widehat{B}=\mathrm{FFT2}(\B)$; then the normal matrix is diagonal with spectrum
$|\widehat{h}|^2+\lambda$, and the per-frequency closed form is
\[
\widehat{X} \;=\; \frac{\overline{\widehat{h}}\,\widehat{B}}{|\widehat{h}|^2+\lambda},
\qquad X=\mathrm{IFFT2}(\widehat{X}),
\]
which coincides with the QSLST solution (see Algorithm~2 of \cite{Fei2025}).
Centering the  PSF prior to applying the FFT is crucial to prevent phase ramps.


\paragraph{NS–Q with \textit{fft} mode (our adaptation)}
Instead of explicit division, we invert the diagonal
$T:=|\widehat{h}|^2+\lambda$ iteratively and pointwise in the frequency domain via Newton–Schulz:
\[
y_{0}=\alpha,\ \ \alpha\approx \frac{2}{\min T+\max T},\qquad
y_{k+1}=y_k\,(2 - T\,y_k),
\]
and set $\widehat{X} = y_K\cdot(\overline{\widehat{h}}\,\widehat{B})$ before inverse FFT.
This quaternion-native \emph{FFT–NS–Q} variant (i) avoids large matrices, (ii) stays entirely in $\mathbb{H}$,
and (iii) converges to the same solution as QSLST for circular BCs, while also supporting alternative diagonalizing transforms (e.g., DCT/DST) for non-circular boundary conditions.


We next quantify the practical efficiency and accuracy of the proposed methods through controlled experiments and real data.

\section{Numerical experiments}\label{sec:experiments}
This section presents the numerical results after implementing our algorithms. We used two execution stacks:
(i) \textbf{MATLAB} (R2023b) with the \textbf{QTFM} toolbox \cite{sangwine2005quaternion} on a laptop with Intel(R) Core(TM) i7-10510U CPU @ 1.80–2.30\,GHz and 16\,GB RAM; and
(ii) \textbf{Python} (3.11) with our quaternion-native library \textbf{QuatIca}~\cite{quatica2025} on a laptop with Apple \textbf{M4 Pro} and 24\,GB RAM.
For transparency, each experiment states the stack used. In brief:
random-matrix pseudoinverse and CUR-based completion ran in \emph{MATLAB+QTFM} (i7-10510U),
while the Lorenz-attractor filtering and FFT-based deblurring ran in \emph{QuatIca/Python} (M4 Pro).
CPU time are compared within the same stack/hardware only; cross-platform timing are reported but not used for speed claims.
Unless otherwise stated, we use the undamped setting with damping parameter $\gamma=1$; any deviations from this default are reported explicitly alongside the corresponding experiment.

For a matrix ${ X}$, the following errors are reported: 
\begin{eqnarray}
&& e_1=||{ X}^{\dagger}{ X}{ X}^{\dagger}-{ X}^{\dagger}||_F,\quad e_2=||{ X}{ X}^{\dagger}{ X}-{ X}||_F,\\
&& e_3=||({ X}^{\dagger}{ X})^H-{ X}^{\dagger}{ X}||_F,\quad e_4=||({ X}{ X}^{\dagger})^H-{ X}{ X}^{\dagger}||_F.
\end{eqnarray}

\begin{exa}
(Random matrices) In this example, we consider random matrices with elements drawn from a standard Gaussian distribution (zero mean and unit variance).  The random matrices used in our experiment are of size $n\times (n+50)$, where $n=100,200,\ldots,800$. The proposed iterative method and the basic algorithm are applied to compute the MP inverse of these matrices. In all experiments, we used 35 iterations for the suggested Algorithm. The run time and accuracy are shown in Figure \ref{fig:runtime-error}. From the computing time figure (left), we observe that the proposed iterative algorithm scales better compared to the baseline. Also, from Figure \ref{fig:runtime-error} (right), we can conclude that the proposed algorithm is not only better in terms of computational time but also provides more accurate results.

\begin{figure}[ht!]
\centering
\subfigure[Runtime vs.\ size]{%
  \includegraphics[width=.49\linewidth]{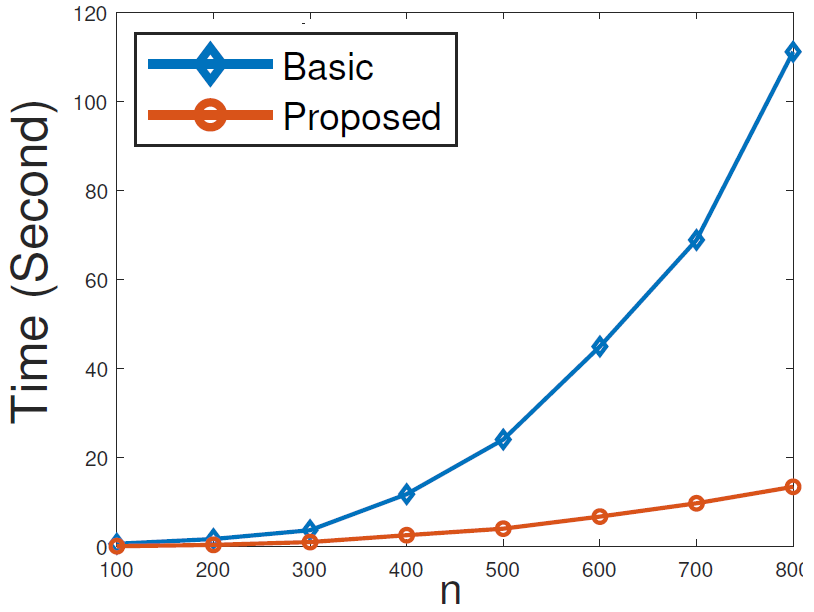}%
}\hfill
\subfigure[MP residual error vs.\ size]{%
  \includegraphics[width=.49\linewidth]{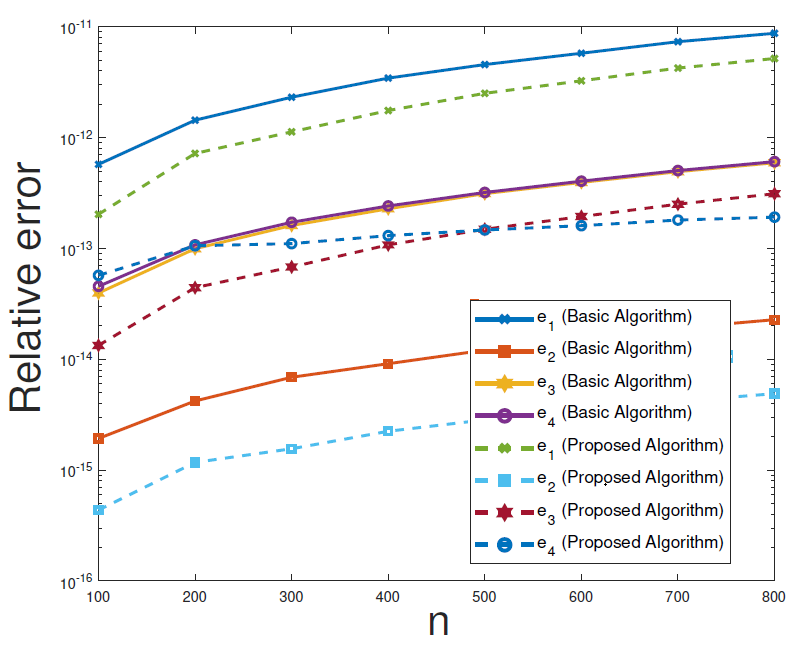}%
}
\caption{\small Comparison of QSVD-based MP inverse (baseline) and the proposed iterative method on random $n\times(n{+}50)$ quaternion matrices.}
\label{fig:runtime-error}
\end{figure}
\end{exa}

\begin{exa}

(Image and video completion)
We evaluate the proposed iterative approach on the Kodak image \textit{kodim16} (size \(512\times768\times3\)), by randomly removing 70\% of the pixels. A rank-60 CUR completion using our iterative MP and the QSVD-based baseline (``basic algorithm'') is run for 25 iterations. Figure~\ref{fig1:ex2} reports runtime and PSNR, and shows the recovered images; the iterative MP yields the same quality at markedly lower wall-time.

   \begin{figure}
\begin{center}
\includegraphics[width=1.1\linewidth]{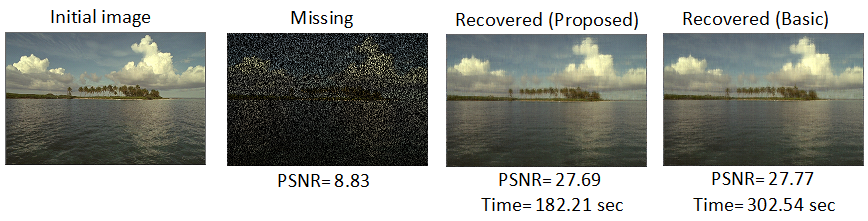}
\caption{Kodim16: original (left), 70\% missing, and reconstructions by QSVD-MP (``basic algorithm") vs.\ iterative MP.}\label{fig1:ex2}

\end{center}
\end{figure}

Color videos are naturally fourth–order tensors (height $\times$ width $\times$ channels $\times$ frames). We use the \textit{flowers} video from YouTube\footnote{\url{https://www.youtube.com/watch?v=bXlQ3Mw4uGc}} (original size $720\times1280\times3\times174$), which we downsample/crop  to $256\times256\times3\times174$. By mapping the RGB channels to the imaginary units $\{\qi,\qj,\qk\}$, we obtain a third–order quaternion tensor $\mathcal{X}\in\mathbb{H}^{256\times256\times174}$, where each frontal slice is a quaternion matrix. We randomly remove 70\% of entries and apply the same CUR–based completion as in the image experiment, comparing our iterative MP against a QSVD–based baseline. To enhance computational throughput, we divide the video sequence into non-overlapping 10-frame blocks, processing them in parallel, with each block undergoing 50 iterations at a target rank of 90. Figure~\ref{fig2:ex2} presents representative frames (original, masked, and reconstructions): the iterative MP method achieves comparable quality to the baseline while reducing wall-clock time by about a half.

  \begin{figure}
\begin{center}
\includegraphics[width=1\linewidth]{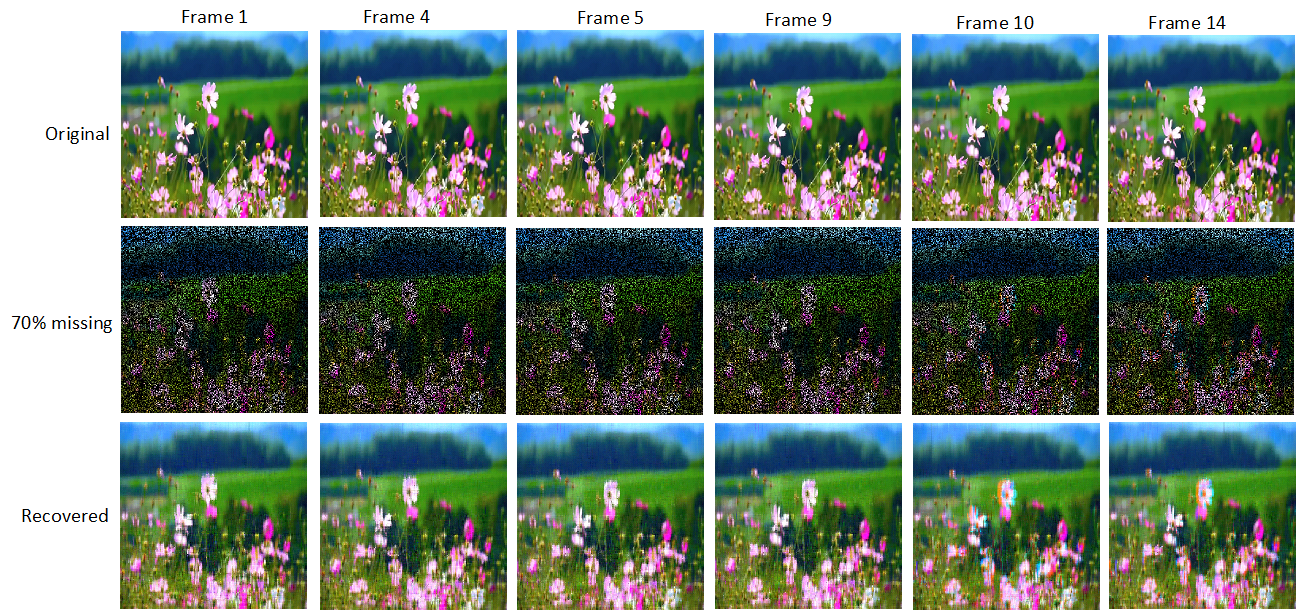}
\caption{Flowers video: original, observed (70\% missing), and reconstructions for selected frames.}\label{fig2:ex2}
\end{center}
\end{figure}

\textbf{Smoothing prior.}
Augmenting CUR with a mild spatial prior can further improve completion quality \cite{ahmadi2023fast}. We therefore incorporate a denoising step into each iteration of \eqref{Step1}: after computing the low-rank estimate ${X}^{(n)}$, we apply a 2-D Gaussian filter (using MATLAB's \verb|imgaussfilt| with a standard deviation of $\sigma=0.5$ and default kernel size) to smooth each video frame. Across iterations, we consistently observed improvements in PSNR and SSIM metrics, along with a reduction in visual artifacts; see Figure~\ref{fig3:ex2}.

  \begin{figure}
\begin{center}
\includegraphics[width=0.8\linewidth]{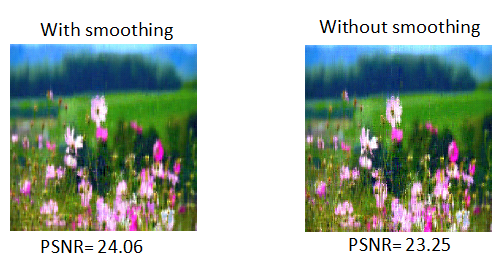}
\caption{\small{The effect of smoothing approaches on improving recovery performance.} }\label{fig3:ex2}
\end{center}
\end{figure}

Our results demonstrate the efficiency of the proposed methodology for processing color videos.

\end{exa}

\begin{exa}(Lorenz–attractor quaternion filtering (NS vs.\ QGMRES))

All results reported in this section were produced with \textbf{QuatIca}, our quaternion-native numerical linear algebra framework, available at \url{https://github.com/vleplat/QuatIca}. QuatIca provides efficient implementations of the solvers used here (e.g., QGMRES and Newton–Schulz in $\mathbb{H}$) and the accompanying experiment scripts. Unless otherwise noted, experiments were run on a laptop with an Apple M4~Pro processor and 24\,GB RAM.

In Table~\ref{tab:lorenz}, we report CPU time, iteration count, and final RelRes for $N\in\{50,75,100,150, 200\}$. 

\begin{table}[ht!]
\centering
\caption{Lorenz–attractor filtering: QGMRES vs.\ NS--Q on the $N\times N$ quaternion system \eqref{eq:lorenz-linear}.}
\label{tab:lorenz}
\begin{tabular}{lcccc}
\hline
$N$ & Method & Iterations & CPU time (s) & RelRes \\
\hline
50 & NS--Q   &  48 &  0.031 & 1.3e-09 \\
    & QGMRES &  50 &  1.080 & 1.2e-15 \\
\hline
75 & NS--Q   &  57 &  0.070 & 6.1e-09 \\
    & QGMRES &  71 &  3.001 & 2.3e-06 \\
\hline
100 & NS--Q   &  60 &  0.128 & 1.3e-08 \\
    & QGMRES & 100 &  8.239 & 1.8e-15 \\
\hline
150 & NS--Q   &  63 &  0.321 & 7.7e-09 \\
    & QGMRES & 150 & 29.550 & 2.1e-15 \\
\hline
200 & NS--Q   &  61 &  0.596 & 8.8e-10 \\
    & QGMRES & 200 & 73.920 & 4.2e-06 \\
\hline
\end{tabular}
\end{table}

\paragraph{Discussion}
Both solvers reach very small final residuals across all sizes. The NS-Q method consistently converges in $\approx$60 iterations and is one to two orders of magnitude faster than unpreconditioned QGMRES in this experiment (e.g., $0.60$\,s vs.\ $73.92$\,s at $N=200$), with CPU time growing gently with $N$.
While QGMRES often reaches near machine precision after $ N $ iterations, the residuals of NS–Q ($ 10^{-8} $ to $ 10^{-10} $) are already well-suited for the downstream applications considered here, with the option to further reduce residuals through additional iterations if required.

\paragraph{Preconditioning}

We observe that incorporating an LU preconditioner, available in \textsc{QuatIca}, can significantly decrease the wall time for QGMRES; however, this does not change the qualitative outcomes for these problems. Additionally, the same LU factors can be utilized to enhance the initialization or scaling of NS--Q, typically leading to a further reduction in its iteration count. Overall, these results support NS--Q as a fast, quaternion-native default solver for square systems of this type, while QGMRES remains a competitive option when high-quality preconditioners are available or when Krylov methods are preferred.

\paragraph{Trajectory visualization (T = 10 s)}
Beyond scalar residuals, we also assess the fidelity of recovered dynamics. We simulate the Lorenz system on $[0, T]$ with $T=10$ seconds and sample uniformly in time such that $N=200$. We then compare the true trajectories $(x(t),y(t),z(t))$ against the outputs reconstructed from the estimated quaternion filter ${w}$ obtained by each solver (QGMRES and NS--Q) in Figure~\ref{fig:lorenz-trajectories_2}. Additionally, we visualize the three 1D signals $x(t)$, $y(t)$, and $z(t)$ vs.\ time in Figure~\ref{fig:lorenz-trajectories}, overlaying:
\begin{itemize}
    \item the observed signal (solid light green);
  \item ground truth (solid black);
  \item NS--Q reconstruction (dashed blue);
  \item QGMRES reconstruction (dotted red).
\end{itemize}

These result show perfect reconstruction of the trajectories obtained by both methods.

\begin{figure}
\begin{center}
\includegraphics[width=0.95\linewidth]{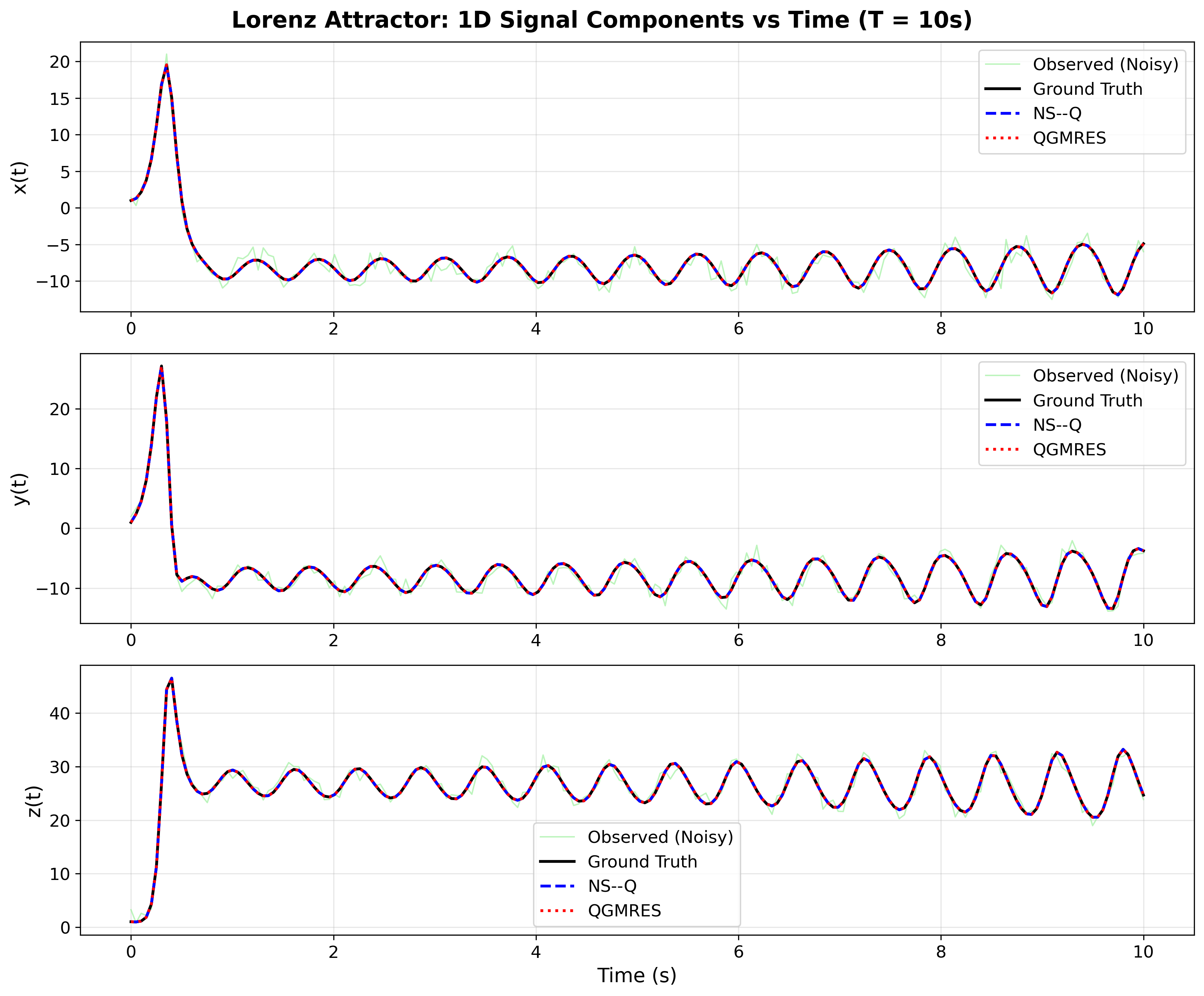}
\caption{True vs.\ reconstructed 1D signal components for the Lorenz system over $T=10$ s. Solid light green: observed signal (noisy), Solid black: ground truth; dashed blue: NS--Q; dotted red: QGMRES.}\label{fig:lorenz-trajectories}
\end{center}
\end{figure}

  \begin{figure}
\begin{center}
\includegraphics[width=0.95\linewidth]{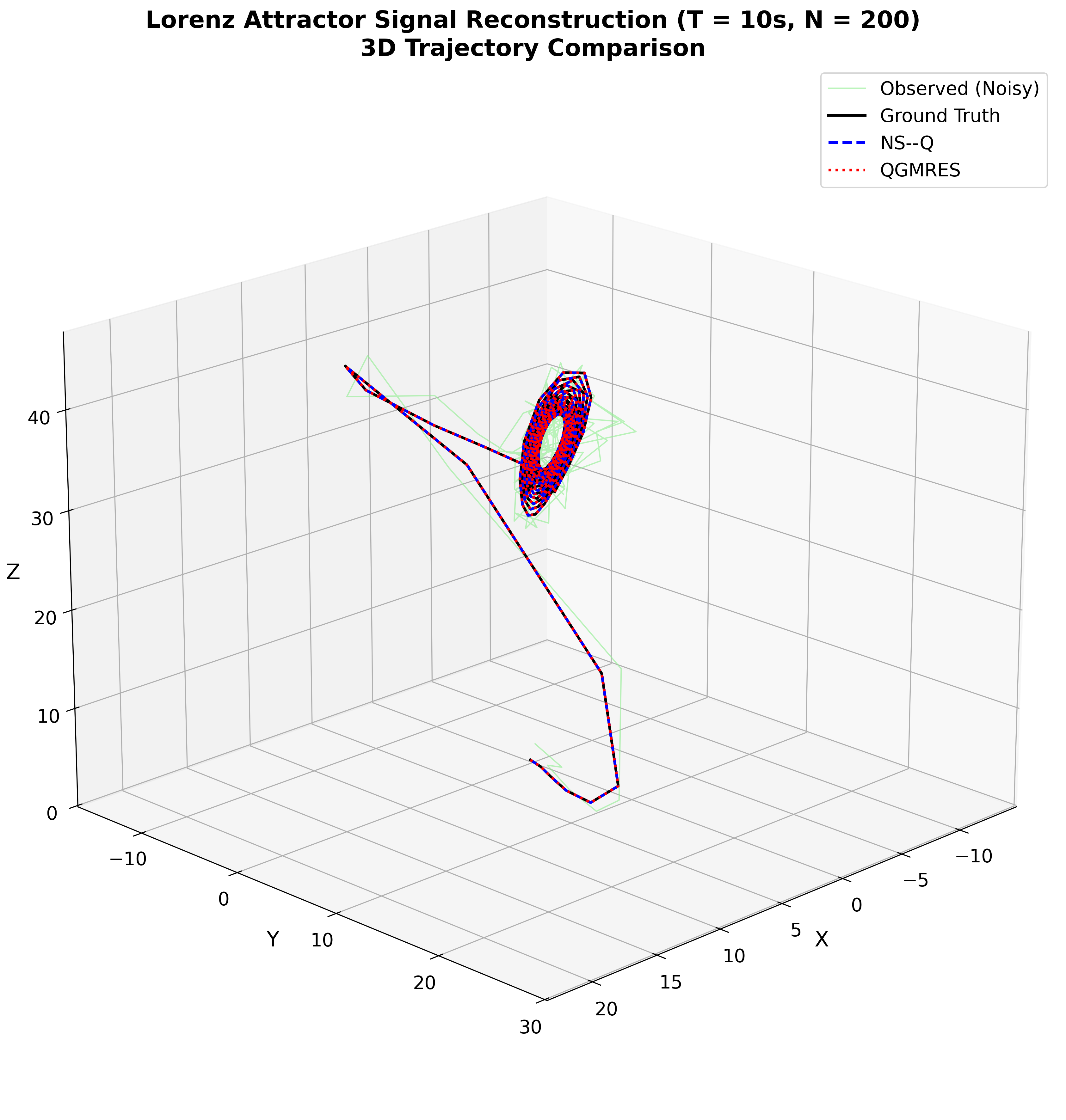}
\caption{True vs.\ reconstructed trajectories for the Lorenz system over $T=10$ s. Solid light green: observed signal (noisy), Solid black: ground truth; dashed blue: NS--Q; dotted red: QGMRES.}\label{fig:lorenz-trajectories_2}
\end{center}
\end{figure}

\end{exa}

\begin{exa}

We compare our quaternion–native \emph{FFT–NS–Q} solver against the \emph{QSLST–FFT} baseline (Algorithm~2 of \cite{Fei2025}, adapted to this setting) on two Kodak color images, \texttt{kodim16} and \texttt{kodim20}. To assess scalability, each method is run on centered $N\times N$ subimages extracted from the originals, increasing $N$ across trials.

Both methods use the same blur kernels (Gaussian PSF with radius \(r=4\), \(\sigma=1.0\); \(9\times9\) kernel) and noise level (30 dB SNR), circular boundary conditions, double precision, and the same FFT backend. For FFT–NS–Q, the per-frequency Newton–Schulz iteration is initialized with \(\alpha \approx 2/(\min T+\max T)\) for \(T=\lvert\widehat{h}\rvert^2+\lambda\) and iterated until the relative normal-equation residual drops below \(10^{-6}\). We select \(\lambda\) via a short grid search per image, and report the resulting \(\lambda_{\text{opt}}\) in Table~\ref{tab:deblur-results}; we also report CPU time (s), PSNR (dB), and SSIM.

\begin{table}[ht!]
\centering
\caption{Nonblind image deblurring: FFT--NS--Q vs.\ QSLST--FFT on $N\times N$ subimages from Kodak images.}
\label{tab:deblur-results}
\begin{tabular}{lccccc}
\hline
$N$ & $\lambda_{\text{opt}}$ & Method & CPU time (s) & PSNR (dB) & SSIM \\
\hline
\multicolumn{6}{l}{\textit{kodim16}} \\
32 & 0.020 & QSLST--FFT &  0.000 &  28.22 & 0.820 \\
    &         & FFT--NS--Q &  0.000 &  28.22 & 0.820 \\
\hline
64 & 0.050 & QSLST--FFT &  0.000 &  28.68 & 0.817 \\
    &         & FFT--NS--Q &  0.000 &  28.68 & 0.817 \\
\hline
128 & 0.050 & QSLST--FFT &  0.001 &  29.37 & 0.786 \\
    &         & FFT--NS--Q &  0.001 &  29.37 & 0.786 \\
\hline
256 & 0.050 & QSLST--FFT &  0.004 &  29.44 & 0.762 \\
    &         & FFT--NS--Q &  0.006 &  29.44 & 0.762 \\
\hline
400 & 0.050 & QSLST--FFT &  0.013 &  29.32 & 0.755 \\
    &         & FFT--NS--Q &  0.017 &  29.32 & 0.755 \\
\hline
512 & 0.050 & QSLST--FFT &  0.025 &  28.89 & 0.751 \\
    &         & FFT--NS--Q &  0.036 &  28.89 & 0.751 \\
\hline
\multicolumn{6}{l}{\textit{kodim20}} \\
32 & 0.020 & QSLST--FFT &  0.000 &  23.66 & 0.728 \\
    &         & FFT--NS--Q &  0.000 &  23.66 & 0.728 \\
\hline
64 & 0.020 & QSLST--FFT &  0.000 &  24.22 & 0.611 \\
    &         & FFT--NS--Q &  0.000 &  24.22 & 0.611 \\
\hline
128 & 0.020 & QSLST--FFT &  0.001 &  24.56 & 0.521 \\
    &         & FFT--NS--Q &  0.001 &  24.56 & 0.521 \\
\hline
256 & 0.050 & QSLST--FFT &  0.004 &  24.69 & 0.560 \\
    &         & FFT--NS--Q &  0.006 &  24.69 & 0.560 \\
\hline
400 & 0.050 & QSLST--FFT &  0.014 &  24.92 & 0.541 \\
    &         & FFT--NS--Q &  0.018 &  24.92 & 0.541 \\
\hline
512 & 0.050 & QSLST--FFT &  0.025 &  24.93 & 0.535 \\
    &         & FFT--NS--Q &  0.033 &  24.93 & 0.535 \\
\hline
\end{tabular}
\end{table}

Across both images and all crop sizes, FFT–NS–Q attains the same PSNR/SSIM as QSLST–FFT while running in comparable time (within a small constant factor). This demonstrates that our quaternion-native solver matches state-of-the-art accuracy and competes on speed without resorting to real embeddings. This parity makes NS–Q a practical drop-in alternative that preserves numerical quality and scales nicely with $N$.

To visually corroborate Table~\ref{tab:deblur-results}, Figure~\ref{fig:deblur-grids-512} shows, for \(N=512\), the originals, the blurred/noisy observations, and the QSLST–FFT and FFT–NS–Q reconstructions; confirming that both recover the scenes to the same visual quality.

\begin{figure}[ht!]
\centering
\subfigure[\textit{kodim16}, $N=512$]{
  \includegraphics[width=0.48\linewidth]{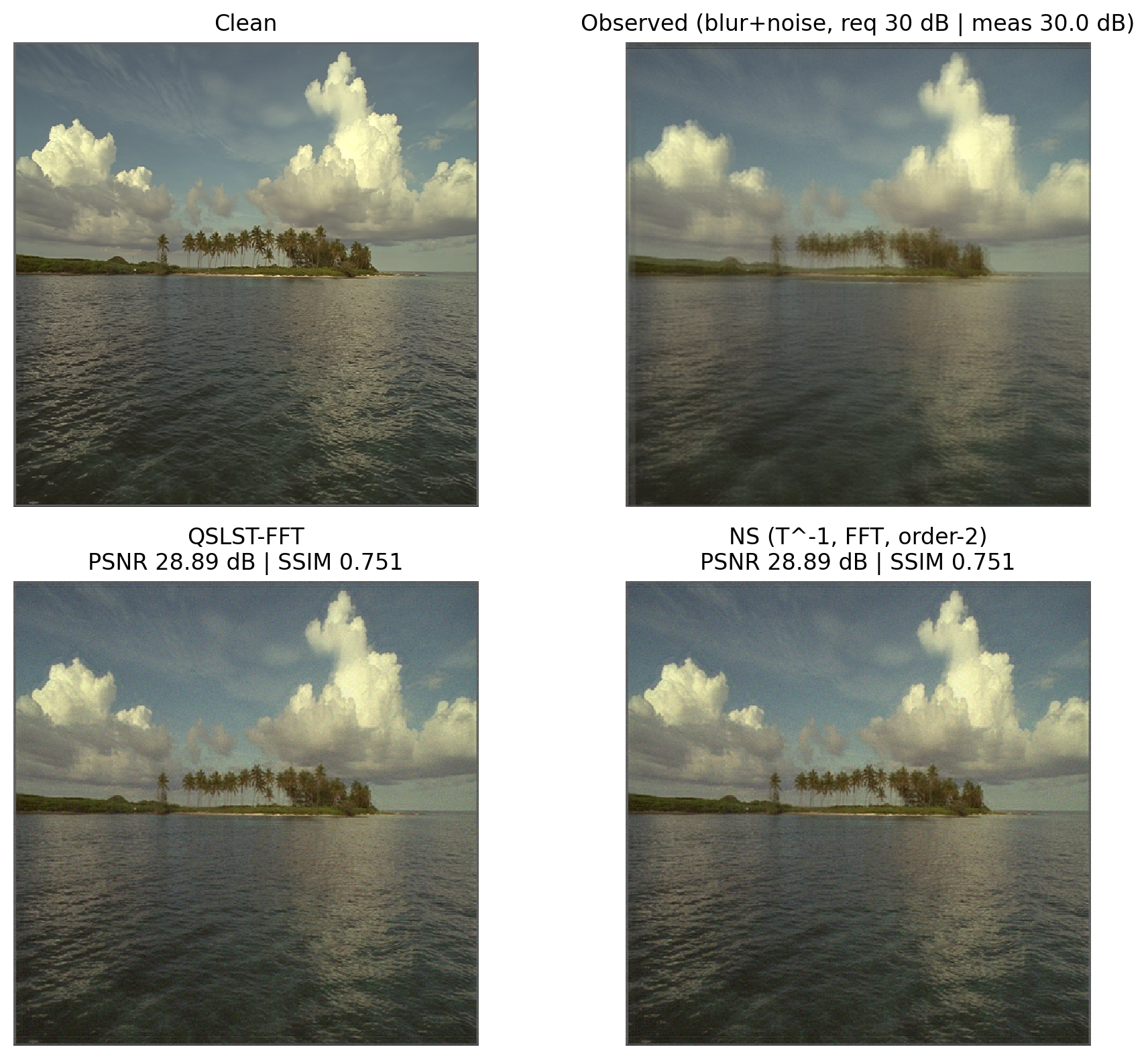}
}
\hfill
\subfigure[\textit{kodim20}, $N=512$]{
  \includegraphics[width=0.48\linewidth]{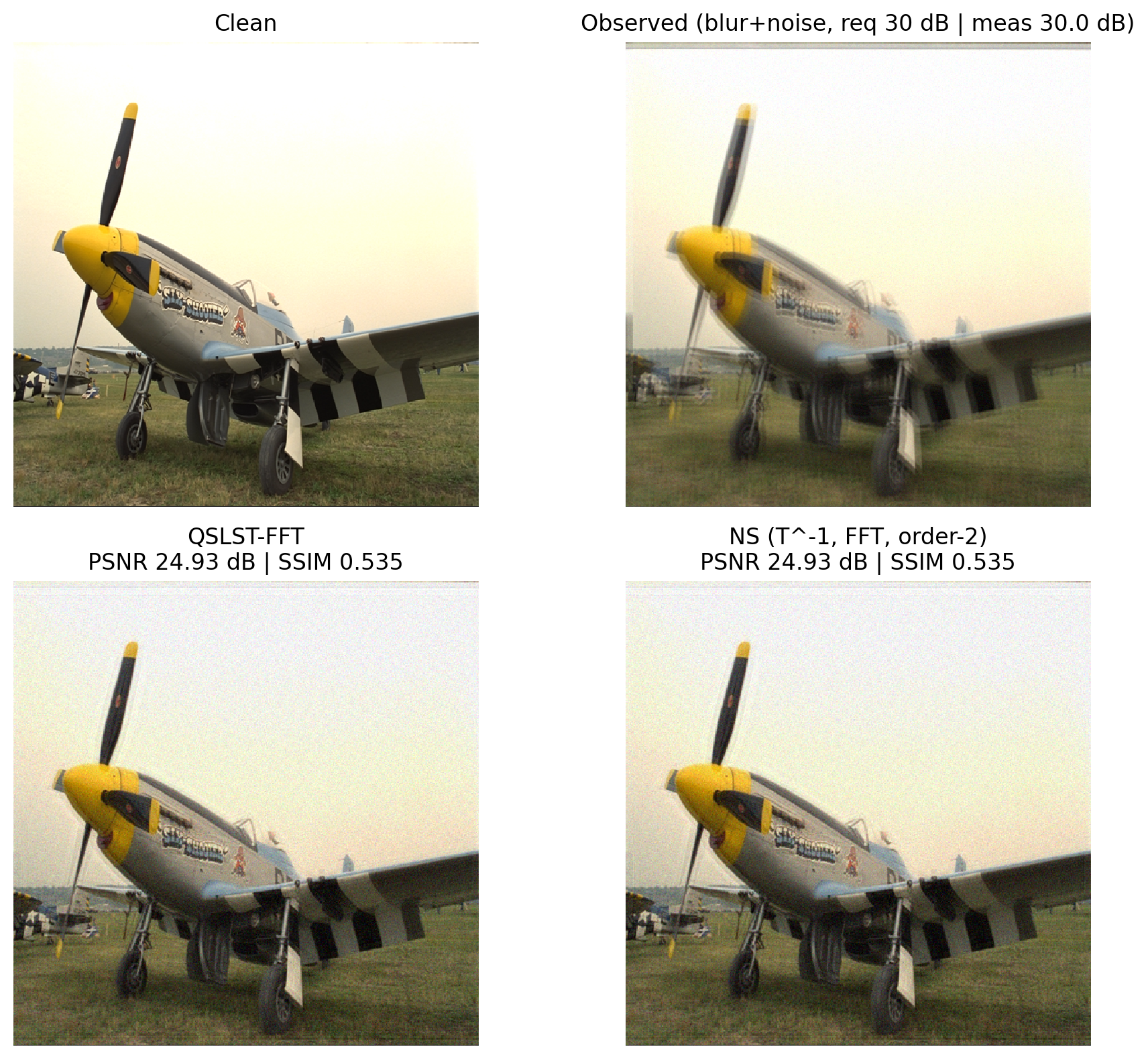}
}
\caption{Nonblind deblurring comparison at \(N=512\): each grid shows, from left to right, the original image, the blurred/noisy observation, the QSLST--FFT reconstruction, and the FFT--NS--Q reconstruction. Both methods produce visually indistinguishable results, consistent with Table~\ref{tab:deblur-results}.}
\label{fig:deblur-grids-512}
\end{figure}

\end{exa}

We will conclude by summarizing the main findings and outlining several promising directions for future work. 
Before drawing conclusions, we first introduce additional iterative schemes in the next section and conduct preliminary validation using our NS-based solvers.


\section{Additional Iterative Schemes and Numerical Comparisons}\label{sec:extended-algos}
\subsection{A Randomized Block Kaczmarz Method}

This section complements our deterministic Newton–Schulz (NS) solvers with a randomized sketch–and–project method in $\mathbb{H}$ for rectangular inputs. 
Unlike NS, which applies a Neumann polynomial to the residual $F_k=I-X_kA$ and enjoys high-order \emph{local} contraction, sketch–and–project (a block Kaczmarz family) makes inexpensive randomized projections onto sketched identity constraints and provides \emph{global} linear convergence in expectation. 
Both methods converge to the same fixed point, $ A^\dagger $, and complement each other effectively: a few randomized steps establish a favorable convergence basin, followed by a single Newton–Schulz (NS) or hyperpower step to achieve quadratic or order-$ p $ convergence behavior. In our experiments, we focus on the column variant ($ m \geq n $); adapting the hybrid approach to the row case is straightforward and reserved for future investigation.

\paragraph{Sketch–and–project updates}
We describe both side-consistent variants.
\begin{enumerate}
\item \textit{Column variant (full column rank, $XA=I_n$).} 
Given $A\in\mathbb{H}^{m\times n}$ with $m\ge n$, draw a thin right sketch $\Omega_k\in\mathbb{H}^{n\times r}$ and form $Y_k=A\,\Omega_k$. 
Project $X_k$ onto $\{X:\,X Y_k=\Omega_k\}$:
\begin{equation}\label{eq:rsp_update_main}
X_{k+1} \;=\; X_k \;+\; \big(\Omega_k - X_k Y_k\big)\,Y_k^\dagger,
\qquad
Y_k^\dagger=(Y_k^H Y_k)^{-1}Y_k^H .
\end{equation}
Since $A^\dagger Y_k=\Omega_k$, $A^\dagger$ is feasible. A relaxed form uses $X_{k+1}=X_k+\gamma(\Omega_k - X_k Y_k)Y_k^\dagger$ with $\gamma\in(0,2)$.
\item \textit{Row variant (full row rank, $AX=I_m$).} 
For $m\le n$, draw a left sketch $S_k\in\mathbb{H}^{m\times r}$, set $Z_k=S_k^H A$, and update
\[
X_{k+1} \;=\; X_k \;+\; Z_k^\dagger\big(S_k^H - Z_k X_k\big),
\qquad
Z_k^\dagger=(Z_k Z_k^H)^{-1}Z_k .
\]
\end{enumerate}

\paragraph{Initialization and seeding}
We employ lightweight initialization strategies that adhere to the side conventions. In the tall case (\( m \geq n \)), we initialize \( X_0 = \alpha A^H \) with a spectrally safe choice of \( \alpha = 1/\|A\|_F^2 \). In the wide case (\( m < n \)), we set \( X_0 = 0 \). This initialization is sufficient as each step involves a random orthogonal projection onto an affine set defined by \( \{X : XA\Omega = \Omega\} \) or \( \{X : S^HAX = S^H\} \).

\paragraph{Work per iteration (and what tends to dominate)}
With block size $r\ll \min\{m,n\}$, the column step performs
\[
\underbrace{A\Omega_k}_{O(mnr)},\quad
\underbrace{\Omega_k-X_kY_k}_{O(nmr)},\quad
\underbrace{Y_k^\dagger\text{ via thin QR or SPD normal equations}}_{O(mr^2+r^3)},\quad
\underbrace{(\Omega_k - X_kY_k)\,Y_k^\dagger}_{O(nmr)}.
\]
The total is \(O(nmr)+O(mr^2+r^3)\); the row variant replaces $A\Omega_k$ by $A^H S_k$ and $Y_k$ by $Z_k$.
In practice, when many randomized steps are required, the repeated thin QR factorizations (column case) or small SPD Gram solves (row case), required for constructing sketched pseudoinverses, emerge as the primary bottleneck affecting runtime (see the benchmark discussion).


\begin{rem}[Implementation details]
For column updates, we offer two interchangeable paths:
\emph{(QR)} compute a thin QR $Y=UR$ and set $Y^\dagger=R^{-1}U^H$; or
\emph{(SPD)} form $G=Y^H Y$ and solve $G Z=Y^H$ with a small quaternion CG micro-solver (add a $10^{-10}I$ ridge; fall back to a few Newton–Schulz steps for $G^{-1}$ if CG stalls).
Row updates always solve $(Z Z^H) W = S^H - Z X$ (same CG+ridge with NS fallback) and set $X \leftarrow X + Z^H W$.
\end{rem}

\paragraph{Monitoring progress without forming $F_k$}
To avoid $XA$ products at every step, we fix an independent test sketch $\Pi\in\mathbb{H}^{n\times s}$ with $s\ll n$ and precompute $A\Pi$ once.
We then stop when
\[
\frac{\|\,\Pi - X_k(A\Pi)\,\|_F}{\|\Pi\|_F}\ \le\ \text{tol},
\]
which estimates $\|I_n-X_kA\|_F$ at cost $O(nms)$ per check.

\begin{algorithm}[ht!]
\RestyleAlgo{ruled}\LinesNumbered
\caption{RSP--Q (column): randomized sketch--and--project for $XA=I_n$}\label{alg:RSP_right_fixed}
\KwIn{$A\in\mathbb{H}^{m\times n}$ full column rank; block $r$; tol; maxit}
\textbf{Init:} $X_0=\alpha A^H$ with $\alpha\in(0,2/\|A\|_2^2)$\;
Draw test sketch $\Pi\in\mathbb{H}^{n\times s}$; precompute $A\Pi$\;
\For{$k=0,1,2,\dots,\text{maxit}$}{
  Draw $\Omega_k\in\mathbb{H}^{n\times r}$\;
  $Y_k \gets A\,\Omega_k$;\quad $\mathcal{R}_k \gets \Omega_k - X_k Y_k$\;
  \tcp{Either thin QR or SPD normal equations to build $Y_k^\dagger$}
  Thin QR: $Y_k=U_k R^{(Y)}_k$;\quad $Y_k^\dagger \gets (R^{(Y)}_k)^{-1} U_k^H$\;
  $X_{k+1} \gets X_k + \mathcal{R}_k\,Y_k^\dagger$\;
  \If{$\|\,\Pi - X_{k+1}(A\Pi)\,\|_F/\|\Pi\|_F \le \text{tol}$}{\textbf{break}}
}
\textbf{return } $X_k$\
\end{algorithm}

\subsubsection{Convergence in expectation.}
Because \eqref{eq:rsp_update_main} is the orthogonal projection of $X_k$ onto an affine set that contains $X_\star=A^\dagger$, one has $\|X_{k+1}-X_\star\|_F\le\|X_k-X_\star\|_F$ deterministically.
When $\Omega_k$ has i.i.d.\ quaternion standard normal entries (or uses leverage-score sampling), the standard sketch–and–project argument yields
\begin{equation}\label{eq:rsp_rate_main}
\mathbb{E}\!\left[\|X_{k+1}-X_\star\|_F^2 \,\middle|\, X_k\right]
\;\le\;
\Bigl(1 - \frac{r\,\sigma_{\min}(A)^2}{\|A\|_F^2}\Bigr)\,\|X_k-X_\star\|_F^2,
\end{equation}
where the quaternion case follows from the real proof using the adjoint rules and submultiplicativity summarized in Section~\ref{sec:prelim}.
This bound is the lens through which we interpret the effect of the block size $r$ in our tests.

\begin{rem}[Practical guidance]
(i) Start with a small block $r$ (e.g., $8$–$32$) and increase it only if convergence stagnates. 
(ii) Applying Simple column/row rescaling to achieve unit $\ell_2$ norms often improves $\sigma_{\min}(A)$ and the convergence rate.
(iii) Importance sampling (approximate leverage scores) accelerates convergence further.
(iv) If a drawn sketch produces a rank-deficient $Y_k$ or $Z_k$, simply redraw a new sketch.
\end{rem}

\paragraph{From global progress to local acceleration: a simple hybrid}
RSP--Q makes inexpensive global progress, whereas NS (and its hyperpower variants) offers rapid local contraction once the residual is small.
We therefore interleave a few randomized steps with an exact right-residual hyperpower correction:
after $T$ RSP steps, compute $F=I_n-XA$ and update $X\leftarrow\bigl(\sum_{i=0}^{p-1}F^i\bigr)X$, with the summation evaluated using  a Paterson–Stockmeyer schedule. 
This “global–then–local’’ pattern is summarized below.

\begin{algorithm}[ht!]
\RestyleAlgo{ruled}\LinesNumbered
\caption{Hybrid RSP--Q + NS (column)}\label{alg:Hybrid_RSP_NS}
\KwIn{$A$; block $r$ (RSP) and order $p$ (NS); cycle length $T$; tol, maxit}
\textbf{Init:} $X\gets \alpha A^H$; draw $\Pi$; precompute $A\Pi$\;
\For{$\ell=0,1,2,\dots$}{
  \tcp{Randomized sketch--and--project phase}
  \For{$t=1$ \KwTo $T$}{
    one step of Algorithm~\ref{alg:RSP_right_fixed} on $X$ (same $\Pi$)\;
  }
  \tcp{One exact NS/hyperpower step on the right residual}
  $F \gets I_n - X A$;\quad $X \gets \Big(\sum_{i=0}^{p-1} F^{\,i}\Big)\,X$ \tcp*{Paterson--Stockmeyer evaluation}
  \If{$\|\,\Pi - X(A\Pi)\,\|_F/\|\Pi\|_F \le \text{tol}$}{\textbf{break}}
}
\textbf{return } $X$
\end{algorithm}

In our experience, $T\in\{5,10\}$ and $p\in\{2,4,8\}$ strike a good balance: the RSP phase keeps each cycle inexpensive ($O(nmr)$ per step), and the infrequent hyperpower correction restores the fast local contraction characteristic of NS once within its basin.

\paragraph{Related work}
Our randomized solver is a quaternion–native adaptation of the sketch–and–project (block Kaczmarz) framework for matrix inversion due to Gower and Richtárik \cite{gower2017randomized}, specialized here to rectangular pseudoinverses in $\mathbb{H}$ with explicit left/right identities (via the adjoint), and coupled with Newton–Schulz for global–then–local acceleration. Unlike \cite{gower2017randomized}, which treats real square matrices (and only gestures at pseudoinverse extensions), we provide the quaternionic formulation, a convergence statement in $\mathbb{H}$, and an implementation that avoids real/complex embeddings.

\subsection{Conjugate Gradient Strikes Back}

We minimize the convex quadratic
\[
f(X)=\tfrac12\|XA-I_n\|_F^2,\qquad X\in\mathbb{H}^{n\times m},
\]
with Frobenius inner product $\langle U,V\rangle_F=\mathrm{Re}\,\mathrm{tr}(U^H V)$.
A short calculation (using the adjoint rules in $\mathbb{H}$) gives
\[
\nabla f(X)=(XA-I_n)A^H,\qquad
\mathcal{H}[\Delta]=\Delta(AA^H).
\]
Thus $f$ is a quadratic with a constant right--multiplication Hessian.

\textit{Why this converges to $A^\dagger$ ? }
The set of minimizers of $f$ is the affine manifold $\{X:\ XA=I_n\}$ (all right inverses).
If we initialize $X_0=\alpha A^H$ and only add directions with right factor $A^H$ (as CGNE--Q will do), then all iterates remain in the subspace $\{Y A^H\}$:
\[
X_k\in\mathrm{span}\{A^H\}\quad\Longrightarrow\quad
X_k=Y_k A^H\;\;\text{for some $Y_k$}.
\]
Within this subspace, the only right inverse is $A^\dagger$:
$X A=I_n$ and $X=Y A^H$ imply $Y(A^HA)=I_n$, hence $Y=(A^HA)^{-1}$ and $X=A^\dagger$.
Therefore, any method that (i) stays in $\mathrm{span}\{A^H\}$ and (ii) drives $\|XA-I_n\|_F\to0$ converges to $A^\dagger$.

\paragraph{Steepest descent step (for intuition)}
Let $R_k:=I_n-X_kA$ be the identity residual and $Z_k:=-\nabla f(X_k)=R_k A^H$.
For a direction $D_k$ and step $\alpha$, the new residual is
\(
(X_k+\alpha D_k)A-I_n= -R_k+\alpha (D_kA).
\)
Minimizing $f(X_k+\alpha D_k)=\tfrac12\|-R_k+\alpha D_kA\|_F^2$ in $\alpha$ yields
\[
\alpha_k=\frac{\langle R_k,\,D_kA\rangle_F}{\|D_kA\|_F^2}.
\]
If we choose steepest descent $D_k=Z_k=R_k A^H$, then
\[
\alpha_k=\frac{\|Z_k\|_F^2}{\|Z_k A\|_F^2}
\quad\text{(since }\langle R_k,\,Z_kA\rangle_F=\|Z_k\|_F^2\text{)},
\]
which already avoids any Kronecker products and uses only multiplies by $A$ or $A^H$.

\paragraph{Conjugate Gradient on the normal equations (CGNE--Q)}
CG improves over steepest descent by building \emph{conjugate} directions w.r.t.\ the Hessian:
$\langle D_i,\,\mathcal{H}[D_j]\rangle_F=0$ for $i\neq j$.
In our setting, this is equivalent to $\langle D_iA,\,D_jA\rangle_F=0$.

\begin{algorithm}[t]
\RestyleAlgo{ruled}\LinesNumbered
\caption{CGNE--Q (column): Conjugate Gradient for $f(X)=\tfrac12\|XA-I_n\|_F^2$}\label{alg:CGNEQ}
\KwIn{$A\in\mathbb{H}^{m\times n}$ (full column rank), tol, maxit}
\textbf{Init:} $X_0=\alpha A^H$ with $\alpha\in(0,2/\|A\|_2^2)$\;
$R_0 \gets I_n - X_0 A$;\quad $Z_0 \gets R_0 A^H$;\quad $D_0 \gets Z_0$ \tcp*{$Z_k=-\nabla f(X_k)$}
\For{$k=0,1,2,\dots,\text{maxit}$}{
  $W_k \gets D_k A$ \tcp*{image of the search direction}
  \If{$\|W_k\|_F=0$}{\textbf{break}}
  $\alpha_k \gets \dfrac{\|Z_k\|_F^2}{\|W_k\|_F^2}$ \tcp*{exact line search in one scalar}
  $X_{k+1} \gets X_k + \alpha_k D_k$\;
  $R_{k+1} \gets R_k - \alpha_k W_k$ \tcp*{$R=I_n-XA$}
  \If{$\|R_{k+1}\|_F \le \text{tol}$}{\textbf{break}}
  $Z_{k+1} \gets R_{k+1} A^H$ \tcp*{new negative gradient}
  $\beta_k \gets \dfrac{\|Z_{k+1}\|_F^2}{\|Z_k\|_F^2}$ \tcp*{Fletcher--Reeves; Polak--Ribi\`ere also admissible}
  $D_{k+1} \gets Z_{k+1} + \beta_k D_k$\;
}
\textbf{return } $X_k$\;
\end{algorithm}

\noindent
\emph{Properties.} (i) All operations are quaternion--native and respect right multiplication order.
(ii) For each iteration, we need one $D_kA$ and one $(\cdot)A^H$ multiply; no factorizations of $A$ are required.
(iii) Because $X_0\propto A^H$ and $D_k$ are linear combinations of $Z_j=R_j A^H$, the iterates stay in $\mathrm{span}\{A^H\}$ and converge to $A^\dagger$ once $\|R_k\|\to0$.

\paragraph{Right preconditioning via a thin sketch}
To cut iteration counts without forming $(AA^H)^{-1}$, we build a Nystr\"om--type approximate inverse using a thin block. Now draw $\Omega\in\mathbb{H}^{n\times r}$ ($r\ll n$) and form $Y:=A\Omega\in\mathbb{H}^{m\times r}$.
Then
\[
M \;:=\; (YY^H)^\dagger \;=\; Y\,(Y^H Y)^{-1}(Y^H Y)^{-1}Y^H
\;\approx\; (AA^H)^{-1}\quad\text{on }\mathrm{range}(Y).
\]
Use a right-preconditioned CGNE: replace $Z_k$ by $\widetilde{Z}_k:=Z_k M$ in the formulas for $\alpha_k,\beta_k$ and in the direction seed ($D_0=\widetilde{Z}_0$).
This only adds multiplies by $Y$/$Y^H$ and small $r\times r$ solves.

\paragraph{Row--rank analogue ($AX=I_m$)}
For full row rank ($m\le n$), the symmetric development applies to $g(X)=\tfrac12\|AX-I_m\|_F^2$ with gradient $A^H(AX-I_m)$ and Hessian $\widetilde{\mathcal{H}}[\Delta]=A^H A\,\Delta$. The matrix--form CG recurrences mirror Algorithm~\ref{alg:CGNEQ}, replacing right multiplies by left ones in the obvious way.

\subsection{Simple numerical tests}

In this section, we provide a simple benchmark comparing the efficiency of all the iterative methods proposed through this paper for computing the MP-inverses of quaternion matrices.
We focus on computing the MP inverse of random quaternion matrices of size $(n + 20) \times n$ with $n \in \
\{20, 50, 100, 150, 200\}$. Unless noted otherwise, we set \texttt{tol}$=10^{-8}$ for NS and CGNE--Q and \texttt{tol}$=10^{-3}$ for RSP--Q/Hybrid to emphasize runtime differences; all methods can be driven to tighter tolerances with proportional increases in iterations.

In Figure~\ref{fig:bench_MP}, we report the relative errors $\frac{\|XA - I_n\|_F}{\|I_n\|_F}$ with respect to size $n$ and the computational time for each tested method, i.e., the NS, the third-order ($p=3$) NS, RSQP (Alg.~\ref{alg:RSP_right_fixed}, the Hybrid RSQP+NS with order $p=4$ (Alg.~\ref{alg:Hybrid_RSP_NS}) and CGNE-Q (Alg.~\ref{alg:CGNEQ}). 
All results reported in this section were produced with \textbf{QuatIca}~\cite{quatica2025}, with the test script available online.

  \begin{figure}
\begin{center}
\includegraphics[width=0.99\linewidth]{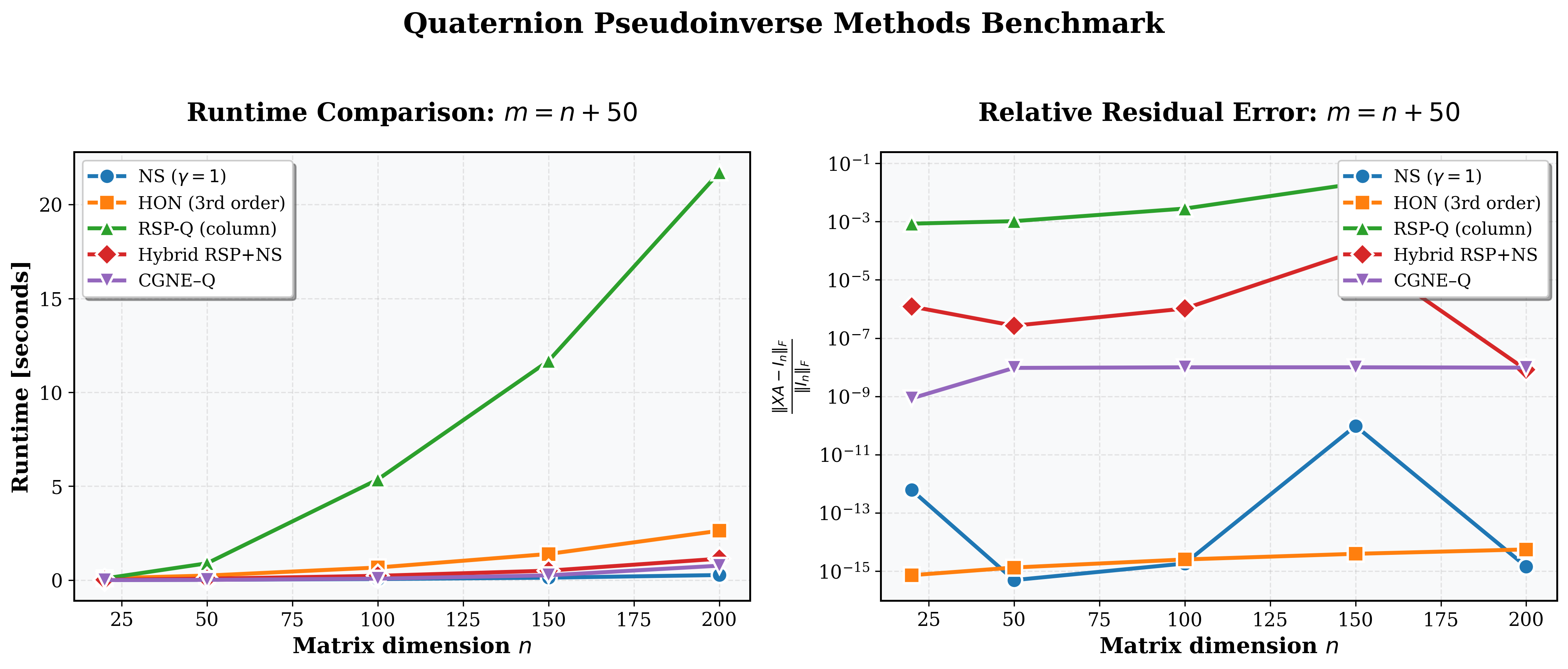}
\caption{Benchmark for proposed iterative methods for computing pseudo-inverses.}\label{fig:bench_MP}
\end{center}
\end{figure}

\paragraph{Cross--method comparison and the RSP--Q bottleneck}
On tall random matrices ($m=n+50$), the damped NS baseline remains the most accurate and fastest overall. The third--order NS attains essentially identical accuracy with slightly higher CPU time, consistent with its higher per--iteration polynomial cost. The hybrid RSP+NS is competitive: randomized steps reduce the right residual, and the injected hyperpower step recovers fast local contraction. CGNE--Q is promising, with predictable per--iteration cost (one multiply by $A$ and one by $A^H$) and strong final relative residuals.

By contrast, \textbf{RSP--Q} achieves acceptable accuracy but at a surprisingly high computation cost, even though sketching reduces the dominant matrix--matrix multiplications to $O(mnr)$. The reason is not due to the multiplies but the per--iteration overhead of forming a sketched pseudo--inverse: repeated thin QR (column) or small SPD Gram solves (row). In our implementation, these micro--solvers (plus stability safeguards: tiny ridge and CG/NS fallbacks) dominate runtime when many randomized steps are needed. We will target this bottleneck in future work via incremental/blocked QR, preconditioned Gram solves, and sketch reuse.

\newpage
\section{Conclusion and future work}\label{sec:conclu}
We proposed quaternion--native Newton--Schulz methods for computing Moore--Penrose pseudoinverses, together with higher--order (hyperpower) variants, and established their convergence directly in $\mathbb{H}$. 
The main highlights include: (i) simple initialization $X_0=\alpha A^H$ with a provable contraction region; (ii) left/right updates that respect noncommutativity and target the appropriate identities for rectangular inputs; and (iii) exact residual recurrences ensuring quadratic (or order-$p$) local convergence.
In addition, we introduced two complementary iterative approaches: a randomized sketch--and--project method (RSP--Q);  and a quaternion conjugate gradient on the normal equations (CGNE--Q) and a simple hybrid that alternates a few RSP--Q steps with an exact hyperpower NS correction.
In controlled synthetic experiments, NS with spectral scaling remains the fastest and most accurate; higher--order NS matches the accuracy at slightly higher runtime; the hybrid attains good accuracy with competitive time; CGNE--Q is promising; and RSP--Q reaches acceptable accuracy but can be slower due to repeated thin QR or small SPD Gram solves needed to form sketched pseudo--inverses.
In our three application case studies (CUR completion, Lorenz filtering, FFT--based deblurring), we utilized only the NS family, which matches the accuracy of QSVD--based MPs  while reducing runtime and memory usage.

There are several directions for future. 
First, extending the analysis to rank-deficient inputs with explicit $\lambda$-regularization and establishing nonasymptotic error bounds would strengthen robustness guarantees. 
Second, exploiting structure (e.g., sparsity, Toeplitz/BCCB, or low-rank factorizations) and batch processing can further reduce quaternion matrix multiplications, accelerating the higher--order schemes. 
Third, generalizing the algorithms to quaternion tensors, split quaternion, dual quaternions and block-structured MPs (e.g.,for multi-view learning) is natural and practically relevant. Fourth, we are currently developing iterative method to compute other types of generalized inverse of quaternions matrices, such as group inverse and Drazin inverse.  
Finally, implementing the algorithms on GPUs and exploring mixed-precision variants are promising directions, given that the iterations are matrix-free and dominated by matrix–matrix products.

\newpage
\section*{Declarations}

\paragraph{Funding}
The authors did not receive support from any organization for the submitted work.

\paragraph{Competing interests}
The authors have no competing interests to declare that are relevant to the content of this article.

\paragraph{Ethics approval}
Not applicable.

\paragraph{Consent to participate}
Not applicable.

\paragraph{Consent for publication}
Not applicable.

\paragraph{Data availability}
All datasets used in this study are publicly available from established repositories; persistent links/citations are provided in the main text and reference list. The synthetic datasets used in our experiments are fully reproducible and can be regenerated using the test scripts included in the open-source \texttt{QuatIca} framework (see Code availability). No new proprietary datasets were generated.

\paragraph{Code availability}
The implementation used in our experiments is available in the open-source \texttt{QuatIca} framework (GitHub; see the software citation in the reference list). A current version can be accessed at \url{https://github.com/vleplat/QuatIca}. The MATLAB implementation of some experiments and algorithms are available at \\ \url{https://drive.google.com/file/d/1HyHtEWiewEXtF508pUaYcdP0g76rZNS3/view?usp=sharing}.


\paragraph{Authors’ contributions}
The authors contributed as follows:
\begin{itemize}
  \item \textbf{Valentin Leplat:} conceptualization; methodology; theoretical development; algorithm design and implementation; software (\texttt{QuatIca}); experiments; applications; writing—original draft.
  \item \textbf{Salman Ahmadi-Asl:} algorithm design; experiments; implementation of some algorithms in MATLAB, applications, data curation; writing—review \& editing.
  \item \textbf{JunJun Pan:} supervision; writing—review \& editing.
  \item \textbf{Ning Zheng:} supervision; writing—review \& editing.
\end{itemize}
All authors read and approved the final manuscript.

\newpage


\bibliographystyle{spmpsci}      
\bibliography{cas-refs}          


\end{document}